\documentclass[a4paper, fleqn]{article}
\usepackage[latin1]{inputenc}
\usepackage[top=2.5cm, left=3cm, right=3cm]{geometry} 
\usepackage{amssymb,amsmath,amsthm,latexsym,amsfonts}
\usepackage{xfrac}
\usepackage{lmodern}
\usepackage{fix-cm}

\usepackage{upgreek}

\usepackage{url}
\usepackage{dsfont}
\usepackage{graphicx}
\usepackage{microtype}
\usepackage{bm}
\usepackage{mathtools}

\usepackage{xcolor}
\usepackage[pdftex,colorlinks=true,urlcolor=blue,citecolor=blue, linkcolor=blue]{hyperref}

\newcommand{\PP}{\mathbb{P}}
\newcommand{\FF}{\mathcal{F}}

\newcommand{\BB}{\mathcal{B}}
\newcommand{\LL}{\mathcal{L}}
\newcommand{\EEE}{\mathcal{E}}
\newcommand{\SSS}{\mathcal{S}}

\newcommand{\MM}{\mathcal{M}}
\newcommand{\RR}{\mathbb{R}}
\newcommand{\EE}{\mathbb{E}}
\newcommand{\QQ}{\mathbb{Q}}

\newcommand{\NN}{\mathbb{N}}
\newcommand{\set}{\SSS^{\infty}\times \LL^2(B)\times \LL^2(\widetilde{\mu})}
\newcommand\abs[1]{\left\vert {#1} \right\vert}
\newcommand{\filt}{\mathbb{F}}

\newcommand{\essinf}{\mathop{\mbox{ess inf}}}
\newcommand{\esssup}{\mathop{\mbox{ess sup}}}
\newcommand{\mal}{\stackrel{\mbox{\hspace{-4pt}\tiny{\tiny$\bullet$}}\hspace{-4pt}}{}}

\newcommand{\footremember}[2]{
    \footnote{#2}
    \newcounter{#1}
    \setcounter{#1}{\value{footnote}}
}

\DeclareRobustCommand{\partialup}{\text{\rotatebox[origin=t]{20}{\scalebox{0.95}[1]{$\partial$}}}\hspace{-1pt}}

\newtheorem{theorem}{Theorem}[section]
\newtheorem{definition}[theorem]{Definition}
\newtheorem{proposition}[theorem]{Proposition}
\newtheorem{lemma}[theorem]{Lemma}
\newtheorem{example}[theorem]{Example}
\newtheorem{corollary}[theorem]{Corollary}
\newtheorem{remark}[theorem]{Remark}

\numberwithin{equation}{section}
\begin{document}


\title{On the monotone stability approach to BSDEs with jumps: Extensions, concrete criteria and examples
\footnote{For the final publication, please refer to \textit{Frontiers in Stochastic Analysis - BSDEs, SPDEs and their Applications}, Springer, 2019, 
\href{https://doi.org/10.1007/978-3-030-22285-7_1}{DOI {10.1007/978-3-030-22285-7\_1}}}
}
\author{
  Dirk Becherer\footremember{HU}{Institut f\"ur Mathematik, Humboldt-Universit\"at zu Berlin,  Unter den Linden 6, D-10099 Berlin, Germany.
 } \and Martin B\"uttner
  \and Klebert Kentia\footremember{GU}{Institut f\"ur Mathematik, Goethe-Universit\"at Frankfurt, D-60054 Frankfurt am Main, Germany.
\newline 
Emails: kentia\,(at)\,aims.ac.za,  becherer\,(at)\,math.hu-berlin.de. 
}
}
\maketitle
\begin{abstract}
We show a concise extension of the monotone stability approach to backward stochastic differential equations (BSDEs)
that are jointly driven  by a Brownian motion and a random measure of jumps, which could be of infinite activity with a non-deterministic and time-inhomogeneous compensator.
The BSDE generator function can be non-convex  and needs not satisfy global Lipschitz conditions  in the jump integrand. 
We  contribute concrete sufficient criteria, 
that are easy to verify, for results on existence and uniqueness of bounded solutions to BSDEs with jumps, and on comparison and a-priori $L^\infty$-bounds.
Several examples and counter examples are discussed to shed light on the scope and applicability of different assumptions,
 and we provide an overview of major applications in finance and optimal control.

\end{abstract}

      

\section{Introduction} \label{sec:intro}
We study bounded solutions $(Y,Z,U)$ to backward stochastic differential equations  with jumps
\[
Y_t = \xi +\int_t^T f_s(Y_{s-},Z_s,U_s)\, {\rm d}s-\int_t^T Z_s\, {\rm d}B_s-\int_t^T \mskip-10mu \int_E U_s(e)\, \widetilde{\mu}({\rm d}s,{\rm d}e)\,,
\]
which are jointly driven by a Brownian motion $B$ and a compensated random measure $\widetilde{\mu}=\mu-\nu^{\PP}$ of some integer-valued random measure $\mu$ 
on a probability space $(\Omega,\cal{F},\PP)$. This  is an extension of the classical BSDE theory on Wiener space towards BSDEs which involve jumps (JBSDEs), that are driven by the compensated random measure  $\widetilde{\mu}$, and do evolve
on non-Brownian filtrations. Such JBSDEs  do involve an additional stochastic integral with respect to the compensated jump measure $\widetilde{\mu}$
whose integrand $U$, differently from $Z$, typically  takes values in an infinite dimensional function space instead of an Euclidean space.

Comparison theorems for BSDEs with jumps require more delicate technical conditions than in the Brownian case, see \cite{BarlesBuckdahnPardoux97,Royer06,CohenElliott10}.
The starting point for our article will be a slight generalization of  
the seminal $({\rm\bf A}_{\mathbf\gamma})$-condition for comparison due to \cite{Royer06}. 
Our first contribution are extensions of comparison, existence and uniqueness results for bounded solutions of 
JBSDEs to the case of  infinite jump activity for a family (\ref{generator1}) of generators, that do not need to be Lipschitz in the $U$-argument.
This shows how the monotone stability 
approach to BSDEs with jumps, pioneered by \cite{Morlais09} for one particular generator, permits for a concise proof 
in a setting, that may be of particular appeal in a pure jump case without a Brownian motion, see Corollary~\ref{cor:JBSDEZzero}. 
While the strong approximation step for this ansatz is usually laborious, we present a compact proof with a $\mathcal{S}^1$-closedness argument and more generality of the generator in the $U$-argument for infinite activity of jumps. 
 To be useful towards applications, our second contribution are sufficient concrete criteria for comparison and wellposedness 
 that are comparably easy to verify in actual examples, 
 because they are formulated in terms of concrete properties for generator functions $f$ from a given family (\ref{generator1}) w.r.t.\ to basically Euclidean arguments,
  instead of assuming inequalities to hold for rather abstract random processes or fields. 
  This is the main thrust for the sufficient conditions of the comparison results 
  in Section~\ref{sec:comp} (see Theorem~\ref{comparetheo} and Proposition~\ref{estimatetheo}, 
  compared to Proposition~\ref{comparegeneral} or the result by \cite{Royer06} and respective enhancements \cite{QuenezSulem13,KrusePopier16,Yao17}) 
  and of the 
wellposedness Theorem~\ref{infinitetheo} (in comparison to Theorem~\ref{infinitetheogeneral}, whose conditions are more general but more abstract). 
A third contribution are the many examples and applications which illustrate the scope and applicability of our results and of the, often  
technical,  assumptions that are needed for JBSDE results in the literature.
Indeed, the range of the imposed combinations of several technical assumptions is often not immediately clear.
We believe that more discussion of examples and counter examples may help to shed light on the scope and the differences of some assumptions prevailing in the literature, 
and might also caution against possible pitfalls.

The approach in this paper can be described in more detail as follows: The comparison results will provide basic a-priori estimates on the $L^\infty$-norm for the $Y$-component of the JBSDE solution.
This step enables a quick intermediate result on existence and uniqueness for JBSDEs with finite jump activity.
To advance from here to infinite activity, we approximate the generator $f$ by a monotone sequence of generators for which solutions do exist, extending 
the monotone stability approach from \cite{Kobylanski00} and (for a particular JBSDE) \cite{Morlais09}. 
For the present paper, the compensator $\nu(\omega,{\rm d}t,{\rm d}e)$ of $\mu(\omega,{\rm d}t,{\rm d}e)$ can be stochastic 
and does not need to be a product measure like $\lambda({\rm d}e)\otimes {\rm d}t$, as it would be natural e.g.\ in a L\'evy-process 
setting,  but it is allowed to be inhomogeneous in that it can vary predictably with $(\omega, t)$. 
In this sense, $\nu$  is only assumed to be absolutely continuous to some reference product measure $\lambda\otimes {\rm d}t$ 
with $\lambda$ being $\sigma$-finite, see equation (\ref{nuzetadensity}). 
Such appears useful, but requires some care in the specification of generator properties in~Section~\ref{sec:prelim}.
For the filtration we assume that $\widetilde{\mu}$ jointly with $B$ (or alone) satisfies the  property of  weak predictable representation for martingales, see (\ref{WPRP}). 
As explained in Example~\ref{example_WPRP}, such setup permits for a range of stochastic dependencies between $B$ and $\widetilde{\mu}$,
which appear useful for modeling of applications, and encompasses many interesting driving noises for jumps in BSDEs; This includes L\'evy processes, Poisson random measures, marked point processes, (semi-)Markov chains or much more general step processes, 
connecting to a wide range of literature, e.g.\ \cite{CohenElliott10,CohenElliot10book,CFJ16,GeissLabart16,GeissS16,GeissS16ejp,BandiniC17}.

The literature on BSDE started with the classical study  \cite{PardouxPeng90} of square integrable solutions to BSDEs
driven solely by Brownian motion $B$ 
under global Lipschitz assumptions.
One important extension concerns generators $f$ which are non-Lipschitz but have quadratic growth in $Z$,  for which  \cite{Kobylanski00} derived bounded solutions by pioneering a monotone stability approach, and
\cite{Tevzadze08} by a fixed point approach. 
Square integrable solutions under global Lipschitz conditions for BSDEs with jumps from a Poisson random measures are first studied  by \cite{TangLi94,BarlesBuckdahnPardoux97}. 
%
There is a lot of development in JBSDE theory recently. See for instance \cite{Bandini15, PapapantoleonEtal16, KrusePopier16,KrusePopier17,OuknineEF17} for 
results under global Lipschitz conditions
 on the generator with respect to
on $(Z,U)$.
In the context of non-Lipschitz generators that are quadratic (also in $Z$, with exponential growth in $U$), JBSDEs have been studied to our knowledge at first by \cite{Morlais09} using a monotone stability approach for a specific generator that is related to
exponential utility, by \cite{ElKarouiMatoussiNgoupeyou16} using a quadratic-exponential semimartingale approach from \cite{BarrieuElKaroui13}, and by \cite{LaevenStadje14} or \cite{KPZ15} with again 
different approaches, relying on duality methods or, respectively, the fixed-point idea of \cite{Tevzadze08} for quadratic BSDEs.
For extensive surveys of the active literature with 
more references, let us refer to \cite{KrusePopier16,Yao17},  who contribute  results on $L^p$-solutions for generators, being monotone in the $Y$-component, that are very general in many aspects.
Their assumptions on the filtrations or generator's dependence on $(Y,Z)$ are for instance more general than ours. 
But the present paper still contributes on other aspects, noted above. For instance, \cite{Yao17} assumes finite activity of jumps and a Lipschitz continuity in $U$. More relations to some other related literature are being explained in many examples throughout the paper, see e.g.\ in Section~\ref{sec:appl}.
Moreover, it is fair to say that results in the JBSDE literature often  involve combinations of many technical assumptions; To understand the scope, applicability and differences of those assumptions, it appears helpful to discuss concrete examples and applications.

The paper is organized as follows. Section~\ref{sec:prelim} introduces the setting and mathematical background. In Sections~\ref{sec:comp}-\ref{sec:exandu}, we prove comparison results and show existence 
as well as uniqueness for bounded solutions to JBSDEs, both for finite and infinite activity of jumps. Last but not least, Section~\ref{sec:appl}  surveys key applications of JBSDEs in finance. We discuss several examples to shed  light on the scope of the results and of the underlying technical assumptions, and connections to the literature.

\section{Preliminaries}\label{sec:prelim}

This section presents the technical framework, sets  notations and discusses key conditions. 
First we recall essential facts on stochastic integration w.r.t.\  random measures and
on bounded solutions for Backward SDEs which are driven jointly by  Brownian motions and a compensated random measure.
 For notions from stochastic analysis not explained here we refer to \cite{JacodShiryaev03,HeWangYan92}.

Inequalities between measurable functions are understood almost everywhere w.r.t.\ 
an appropriate reference measure, typically $\PP$ or $\PP\otimes {\rm d}t$.
Let $T<\infty$ be a finite time horizon and $(\Omega, \mathcal{F}, (\mathcal{F}_t)_{0\leq t\leq T},\mathbb{P})$ a filtered probability space with a filtration $\filt= (\mathcal{F}_t)_{0\leq t\leq T}$ satisfying the usual conditions of right continuity and completeness, assuming $ \mathcal{F}_T= \mathcal{F}$ and $ \mathcal{F}_0$ being trivial (under $\mathbb{P}$);
Thus we can and do take all semimartingales to have right continuous paths with left limits, so-called c\`adl\`ag paths.
Expectations (under $\mathbb{P}$) are denoted by $\mathbb{E}=\mathbb{E}_{\mathbb{P}}$.
We will denote by ${\sf \bf A}^{T}$ the transpose of a matrix ${\sf \bf A}$ and simply write $\bm{xy}:={\bm x}^T{\bm y}$ for the scalar product for two vectors ${\bm x},{\bm y}$ of same dimensionality.
Let $H$ be a separable Hilbert space and denote by $\mathcal{B}(E)$ the Borel $\sigma$-field of $E:=H\backslash \{ 0\}$, e.g.\ $H=\RR^l$, $l\in \NN$ or $H=\ell^2\subset \RR^{\NN}$.
Then $(E,\mathcal{B}(E))$ is a standard Borel space.
In addition, let $B$ be a $d$-dimensional Brownian motion. Stochastic integrals of a vector valued predictable process $Z$ w.r.t.\  a semimartingale $X$, e.g.\  $X=B$, of the same dimensionality are scalar valued semimartingales starting at zero and denoted by
$\int_{(0,t]} Z {\rm d}X=\int_0^t Z {\rm d}X=Z\mal X_t$ for $t\in [0,T]$.
 The \textit{predictable} $\sigma$-field on $\Omega \times [0,T]$ (w.r.t.\ $ (\mathcal{F}_t)_{0\leq t\leq T}$)
is denoted
by $\mathcal{P}$ and $\widetilde{\mathcal{P}}:=\mathcal{P}\otimes \BB (E)$ is the respective $\sigma$-field on $\widetilde{\Omega}:=\Omega \times [0,T] \times E$.

Let $\mu$ be an integer-valued random measure with compensator $\nu =\nu^{\PP}$  (under $\PP$) 
which is taken to be absolutely continuous to $\lambda \otimes {\rm d}t$ for a $\sigma$-finite measure $\lambda$ 
on $(E,\mathcal{B}(E))$ satisfying $\int_E 1\wedge |e|^2 \lambda ({\rm d}e)<\infty$ with some $\widetilde{\mathcal{P}}$-measurable, bounded and non-negative density $\zeta$, such that
\begin{equation}
 \label{nuzetadensity} 
\nu ({\rm d}t,{\rm d}e)=\zeta (t,e)\, \lambda ({\rm d}e)\, {\rm d}t = \zeta_t\, {\rm d}\lambda \,{\rm d}t,
\end{equation} with $0\leq \zeta (t,e)\leq c_{\nu}$ $\PP\otimes \lambda\otimes {\rm d}t$-a.e.\  for some constant $c_{\nu}>0$.
 Note that $L^2(\lambda)$  and $L^2(\zeta_t{\rm d}\lambda)$ are separable Hilbert spaces since $\lambda$ (and  $\lambda_t:=\zeta_t\, {\rm d}\lambda$) is $\sigma$-finite and $\mathcal{B}(E)$ is finitely generated.
Since the density $\zeta$ can vary with $(\omega,t)$, the compensator $\nu$ can be time-inhomogeneous and stochastic. 
Such permits for a richer dependence structure for $(B,\widetilde{\mu})$;\ For instance, the intensity and distribution of jump heights could vary
according to some diffusion process. Yet,  it also brings a few technical complications, e.g.\  function-valued integrand processes $U$ from $\mathcal{L}^2(\widetilde{\mu})$ (as defined below) for the JBSDE need not take values in one given  $L^2$-space (for a.e.\  $(\omega,t)$), like e.g.\ $L^2(\lambda)$ if $\zeta\equiv1$, and the 
specifications of the domain and of the measurability for the generator functions should take account of such. 

For stochastic integration w.r.t.\   $\widetilde{\mu}$ and $B$  we define sets of $\mathbb{R}$-valued processes
\begin{align*}
\mathcal{S}^{p}&:=\mathcal{S}^{p}(\PP ):= \Big\{Y\,  \mbox{ c\`adl\`ag} \, :\, |Y |_{p}:= \Big\lVert \sup\limits_{0\leq t\leq T} |Y_t|\, \Big\rVert_{L^{p}(\PP)} < \infty  \Big\} \quad\mbox{for } p \in [1,\infty]\,,\\
\mathcal{L}^2(\widetilde{\mu}) &:= \Big\{ U\,   \mbox{ } \widetilde{\mathcal{P}}\mbox{-measurable} \, :\,   \lVert U \rVert_{\mathcal{L}^2(\widetilde{\mu})}^2 := \mathbb{E}\Big( \int_0^T \mskip-10mu \int_E |U_s(e)|^2\, \nu({\rm d}s,{\rm d}e) \Big) < \infty \Big\}\,,
\end{align*}
and the set of  $\mathbb{R}^{d}$-valued processes
\begin{align*}
\mathcal{L}^2(B) &:= \Big\{ \theta\, \,\mbox{ } \mathcal{P}\mbox{-measurable} \, :\,  \lVert \theta \rVert_{\mathcal{L}^2(B)}^2 := \mathbb{E}\Big( \int_0^T \lVert \theta_s \rVert^2\, {\rm d}s\, \Big) <\infty \Big\},
\end{align*}
where $\widetilde{\mu}=\widetilde{\mu}^{\PP}=\mu -\nu$ denotes the compensated measure of $\mu$ (under $\PP$).
Recall that for any predictable function $U$,  $\EE (|U|* \mu_T)=\EE (|U|*\nu_T)$ by the definition of a compensator. If $(|U|^2*\mu)^{1/2}$ is locally integrable, then $U$ is integrable w.r.t.\  $\widetilde{\mu}$, and $U*\widetilde{\mu}$ is defined as the purely discontinuous local martingale with jump process $\big(\int_E U_t(e)\, \mu(\{ t\} ,{\rm d}e)\big)_t$ by \cite[Def.II.1.27]{JacodShiryaev03} noting that $\nu$ is absolutely continuous to $\lambda \otimes {\rm d}t$.
For $Z\in \LL^2(B)$ and $U\in \LL^2(\widetilde{\mu})$ we recall that $Z\mal B$ and $U*\widetilde{\mu}= (U*\widetilde{\mu}_t)_{0\le t\le T}$ with $U*\widetilde{\mu}_t =\int_0^t \int_E U_s(e)\, \widetilde{\mu}({\rm d}s,{\rm d}e)$ are square integrable martingales by \cite[Thm.II.1.33]{JacodShiryaev03}. For $Z,Z'\in \LL^2(B)$ and $U,U'\in \LL^2(\widetilde{\mu})$ we have for the predictable quadratic covariations
that $\big\langle U*\widetilde{\mu},U'*\widetilde{\mu}\big\rangle_t=\int_0^t \mskip-5mu \int_E U_s(e)U_s'(e)\, \nu ({\rm d}s,{\rm d}e)$ by \cite[Thm.II.1.33]{JacodShiryaev03}, $\big\langle \int Z\, {\rm d}B, \int Z'\, {\rm d}B\rangle_t =\int_0^t Z_s^{T} Z_s'\, {\rm d}s$ and $\big\langle \int Z\, {\rm d}B, U*\widetilde{\mu}\big\rangle_t=0$ by \cite[Thm.I.4.2]{JacodShiryaev03}.
\\
We denote the space of square integrable martingales by $\mathcal{M}^2$ and its norm by $\lVert \cdot \rVert_{\mathcal{M}^2}$ with $\lVert M\rVert_{\mathcal{M}^2}=\EE (M_T^2)^{\sfrac{1}{2}}$. We recall \cite[Thm.10.9.4]{HeWangYan92} that the subspace of BMO($\PP$)-martingales ${\rm BMO}(\PP )$ contains  any square integrable martingale $M$ with uniformly bounded jumps and bounded 
conditional expectations for increments of the quadratic variation process:
\[
 \sup\limits_{0\leq t\leq T} \left\lVert\EE \big((M_T-M_t)^2\,|\,\FF_t\big)\right\rVert_{L^\infty(\PP)} = \sup\limits_{0\leq t\leq T} \left\lVert\EE \big(\langle M\rangle_T-\langle M\rangle_t\, |\,\FF_t\big)\right\rVert_{L^\infty(\PP)} \le {\rm const} < \infty.
\]
We will assume that the continuous martingale $B$ and the compensated measure $\widetilde{\mu}$ of an integer-valued random measure $\mu$ (or $\widetilde{\mu}$ alone, see
Example~\ref{example_WPRP}.1 and Corollary~\ref{cor:JBSDEZzero} with trivial $B=0$)
jointly have the weak predictable representation property (weak PRP) w.r.t.\  the filtration $(\mathcal{F}_t)_{0\leq t\leq T}$, in 
that every square integrable martingale $M$ has a (unique) representation, i.e.
\begin{equation}
\label{WPRP}
\text{for all} \:M\in {\cal{M}}^2 \;\text{there exists } Z,U \;\text{such that }\; M=M_0 +\int Z\, {\rm d}B+U*\widetilde{\mu}\,,
\end{equation}
with (unique) $Z \in\mathcal{L}^2(B)$ and $U\in \mathcal{L}^2(\widetilde{\mu})$.
Let us note that in the literature \cite[III.\S4c]{JacodShiryaev03} or \cite[XIII.\S2]{HeWangYan92} the weak representation property is defined as a decomposition like $(\ref{WPRP})$ for any local martingale $M$ with integrands $Z$, $U$ being integrable in the sense of local martingales. Such clearly implies our formulation above. Indeed, for a (locally) square integrable martingale $M$  in such a decomposition  both integrands must be at least locally square integrable and
$\langle M\rangle= \int |Z|^2\,{\rm d}t+|U|^2*\nu$ by strong orthogonality of the stochastic integrals. Then $E[\langle M\rangle_T]<\infty$ implies that $Z$, $U$ are in the respective $\mathcal{L}^2$-spaces. We exemplify how  (\ref{WPRP}) connects with a wide literature.

\begin{example}
\label{example_WPRP}
The  weak predictable  representation property  (\ref{WPRP}) holds in the cases below. Cases 1.-4.\ are well known from classical theory \cite{HeWangYan92}, see \cite[Example~2.1]{Becherer06} for details.
\begin{enumerate}\item[1.] Let $X$ be a L\'evy process with $X_0=0$ and predictable characteristics $(\alpha ,\beta ,\nu )$ (under $\PP$). 
Then the continuous martingale part $X^c$ (rescaled to a Brownian motion if $\beta\neq0$, or being trivial if $\beta=0$) and  the compensated jump measure $\widetilde{\mu}^X=\mu^X-\nu$ of $X$ have the weak PRP w.r.t.\  the usual filtration  $\filt^{X}$ generated by $X$. 
 An example for a  L\'evy process of infinite activity is the  Gamma process.
One can add that weak PRP even holds  in the sense of Thm~III.4.34 from \cite{JacodShiryaev03} for the more general class of
 PII-processes with independent increments. This class encompasses the more familiar L\'evy processes without requiring time-homogeneity or stochastic continuity.
\item[2.] Assume that $B$ and $\widetilde{\mu}$ satisfy (\ref{WPRP})  under $\PP$. 
Let $\PP'$ be an equivalent probability measure with density process $Z$. 
Then the Brownian motion $\nolinebreak{B':=B-\int (Z_-)^{-1}\, {\rm d}\langle Z,B\rangle}$ and $\widetilde{\mu}':=\mu-\nu^{\PP'}$ have the weak PRP (\ref{WPRP}) 
also w.r.t.\  $\PP'$ under the same filtration. This offers
plenty  of  scope  to construct examples where $W$ and $\widetilde{\mu}$ are not independent, based on other examples.
\item[3.] Let $B$ be a Brownian motion independent of a step process $X$ (in the sense of~\cite[Ch.\ 11]{HeWangYan92}). 
Then $B$ and $\widetilde{\mu}$, the compensated measure of the jump measure $\mu^X$ of $X$, 
have the weak PRP w.r.t.\  the usual filtration generated by $X$ and $B$. An example for a step process is a multivariate (non-explosive) point process, as appearing in \cite{CFJ16}.
\item[4.] A (semi-)Markov chain $X$, possibly time-inhomogeneous, is a step process.
Thus weak PRP (\ref{WPRP})  holds for a filtration generated by a Brownian motion and an independent Markov chain,
relating later results to literature \cite{CohenElliott10,ConfortolaFuhrman14,BandiniC17} on BSDEs driven by compensated random measures of the respective pure-jump (semi-)Markov processes.
Markov chains $X$ on countable state spaces can be chosen \cite{CohenElliott10} to take values in the set of unit vectors $\{e_i:i\in \NN\}$ of the sequence space $\ell^2\subset \RR^{\NN}$, with jumps $\Delta X$ taking values $e_i-e_j$, $i,j\in \NN$.
\item[5.] The pure jump martingale $U *\widetilde{\mu}$ (for $U \in \mathcal{L}^2(\widetilde{\mu})$) may be written as a series of mutually orthogonal martingales. More precisely, assume that the compensator coincides with the product measure $\lambda \otimes {\rm d}t$, i.e.\ $\zeta=1$. Let $(u^n)_{n\in \NN}$ be an orthonormal basis (ONB) of the separable Hilbert 
space $L^2(\lambda )$ with scalar product $\langle u,v\rangle :=\int_E u(e)v(e)\, \lambda ({\rm d}e)$. Let $U_t=\sum_{n\in \NN} \langle U_t,u^n\rangle u^n$ be the basis expansion of $U_t$ for $U\in \mathcal{L}^2(\widetilde{\mu})$, $t\le T$. Then it holds (in ${\cal M}^2$)
\begin{equation} \label{repformula}
U*\widetilde{\mu}= \sum_{n\in \NN} \int_0^{\cdot} \langle U_t,u^n\rangle \int_E u^n(e)\, \widetilde{\mu}({\rm d}t,{\rm d}e)
=:  \sum_{n\in \NN} \int_0^{\cdot} \alpha^n_t \,{\rm d}L^n_t=\sum_{n\in \NN} \alpha^n \mal L^n,
\end{equation}
for  $\alpha_t^n:=\langle U_t,u^n\rangle$ and $L^n:=u^n*\widetilde{\mu}$. Indeed, setting $F_t^n:=\sum_{k=1}^n \langle U_t,u^k\rangle u^k=\sum_{k=1}^n \alpha_t^k u^k$
one sees that
$\|\sum_{k=1}^\infty |\alpha^k|^2\|_{L^1(\PP\otimes {\rm d}t)} \le \lVert U\rVert_{\mathcal{L}^2(\widetilde{\mu})}^2 <\infty$.
By dominated convergence one obtains as $n\to \infty$
\begin{equation*}
\lVert F^n-U \rVert_{\mathcal{L}^2(\widetilde{\mu})}^2 = \EE \Big( \int_0^T \mskip-10mu \int_E |F_t^n(e)-U_t(e)|^2\lambda ({\rm d}e)\,{\rm d}t\Big) = 
\EE \Big( \int_0^T \sum_{k=n+1}^\infty |\alpha_t^k|^2\, {\rm d}t \Big) \rightarrow 0. 
\end{equation*}
Isometry implies that the stochastic integrals $F^n*\widetilde{\mu}$ converge to $U *\widetilde{\mu}$ in ${\cal M}^2$,
proving (\ref{repformula}).

In particular, we see how the  PRP $(\ref{WPRP})$ w.r.t.\  a random measure can be rewritten as series of ordinary stochastic integrals w.r.t.\  scalar-valued strongly orthogonal martingales $L^n$, which are in fact L\'evy processes with deterministic characteristics $(0,0,\int u^n(e)\, \lambda ({\rm d}e))$.
In this sense, the general condition $(\ref{WPRP})$ links well with results on PRP and BSDEs for L\'evy processes in~\cite{NualartSchoutens00,NualartSchoutens01}
who study a specific Teugels martingale basis consisting of compensated power jump processes for  L\'evy processes which satisfy exponential moment conditions. For a systematic analysis of related  PRP results, comprising  general L\'evy processes, see \cite{DiTellaEngelbert13,DiTellaE15}.
\item[6.] Previous arguments could extend to the general case with $\zeta \not\equiv 1$ in $(\ref{nuzetadensity})$. To this end, suppose  $U^n$ to be in ${\cal L}^2(\widetilde{\mu})$ such that for all $t\le T$ the sequence
  $(U_t^n)_{n\in \NN}$ is ONB of $L^2(\lambda_t)$  for ${\rm d}\lambda_t=\zeta_t {\rm d}\lambda$ with scalar product $\langle u,v \rangle_t :=\int_E u(e)v(e)\, \zeta (t,e)\lambda ({\rm d}e)$. Analogously to case 5.\ above,  with $\alpha_t^n:=\langle U_t,U_t^n\rangle_t$ and $L^n:=U^n*\widetilde{\mu}$ one gets equalities of martingales (in ${\cal M}^2$)
\begin{equation*}
U*\widetilde{\mu} = \sum_{n\in \NN} \int_0^{\cdot} \langle U_t,U_t^n\rangle_t \int_E U_t^n(e)\, \widetilde{\mu}({\rm d}t,{\rm d}e)
=: \sum_{n\in \NN} \alpha^n \mal L^n\,.
\end{equation*}
\end{enumerate}
\end{example}

To proceed, we now define a solution of the Backward SDE with jumps to be a triple $(Y,Z,U)$ of processes in the space $\mathcal{S}^{p} \times \mathcal{L}^2(B) \times \mathcal{L}^2(\widetilde{\mu})$ for a suitable $p\in (1,\infty ]$ that satisfies
\begin{equation}
Y_t = \xi + \int_{t}^{T} f_s(Y_{s-},Z_s,U_s)\, {\rm d}s-\int_{t}^{T}Z_s\, {\rm d}B_s -
\int_{t}^{T}\mskip-10mu \int_{E} U_s(e)\, \widetilde{\mu}({\rm d}s,{\rm d}e),\quad 0\leq t\leq T, \label{BSDE}
\end{equation} 
for given data $(\xi ,f)$, consisting of a $\FF_T$-measurable random variable $\xi$ and a  generator function $f_t(y,z,u)=f(\omega ,t,y,z,u)$. The values $p$ will be specified below in the respective results, although a particular focus will be on bounded BSDE solutions (i.e.\ $p=\infty$).
Because we permit $\nu$ to be time-inhomogeneous with a bounded but possibly non-constant density $\zeta$ in (\ref{nuzetadensity}), it does not hold in general that $U_t$ takes values a.e.\  in one space $L^2(\lambda)$  for $U\in  \mathcal{L}^2(\widetilde{\mu})$.
 This requires some extra consideration about the domain of definition and measurability of $f$, as the generator function $f$
needs to be defined 
for $u$-arguments from a suitable domain, which cannot be some fixed $L^2$-space  in general (and needs to be larger than $L^2(\lambda)$), as integrability of $u=U_t(\omega,\cdot)$ over $e\in E$ may vary with $(\omega,t)$. 
On suitable larger domains, one typically may have to admit for $f$ to attain non-finite values.
To this end, let us  denote by $L^0(\mathcal{B}(E),\lambda )$ the space of all $\mathcal{B}(E)$-measurable functions with the topology of convergence in measure and define
\begin{equation}\label{utnorm}
|u-u'|_t:=\Big( \int_E |u(e)-u'(e)|^2\, \zeta (t,e)\, \lambda ({\rm d}e) \Big)^{\frac{1}{2}},
\end{equation}
 for functions $u,u'$ in $L^0(\mathcal{B}(E),\lambda )$.
Terminal conditions $\xi$ for BSDE considered in this paper will be taken to be square integrable  $\xi \in L^{2}(\FF_T)$ and often even as bounded $\nolinebreak{\xi \in L^{\infty}(\FF_T)}$. Generator functions $f: \Omega \times [0,T]\times \mathbb{R} \times \mathbb{R}^d \times L^0(\mathcal{B}(E),\lambda) \rightarrow \overline{\mathbb{R}}$ are always taken to be $\mathcal{P}\otimes \BB (\RR^{d+1})\otimes \BB (L^0(\BB (E),\lambda ))$-measurable. 
Main  Theorems $\ref{comparetheo}$ and $\ref{infinitetheo}$ are derived for families of generators having the form
\begin{align}
f_t(y,z,u):= \widehat{f}_t(y,z) + \int_A g_t(y,z,u(e),e) \zeta (t,e) \lambda ({\rm d}e)\, &\quad\mbox{(where finitely defined)}  \label{generator1}
\end{align}
and $f_t(y,z,u):=\infty$ elsewhere, or more specially (for a $g$-component not depending on $y,z$)
\begin{align}
f_t(y,z,u):= \widehat{f}_t(y,z) + \int_A g_t(u(e),e)\, \zeta (t,e)\, \lambda ({\rm d}e)\,& \quad \mbox{(where finitely defined)}  \label{generator}
\end{align}
and $f_t(y,z,u):=\infty$ elsewhere,
for a $\mathcal{B}(E)$-measurable set $A$ and component functions $\widehat{f}$, $g$ where $\nolinebreak{\widehat{f}: \Omega  \times [0,T] \times  \mathbb{R}^{1+d} \rightarrow  {\RR}}$ is $\mathcal{P}\otimes \mathcal{B}(\mathbb{R}^{d+1})$-measurable and $\nolinebreak{g: \Omega \times [0,T]\times  \mathbb{R}^{1+d} \times \mathbb{R} \times E \rightarrow \mathbb{R}}$ is $\mathcal{P}\otimes \mathcal{B}(\mathbb{R}^{d+2}) \otimes \mathcal{B}(E)$-measurable.
Clearly statements for generators of the form $(\ref{generator1})$ are also true for those of the (more particular) form $(\ref{generator})$.
(In)finite activity relates to generators with $\lambda (A)<\infty$ (respectively $\lambda (A)=\infty$).
A simple but useful technical Lemma clarifies how we can (and always will) choose a bounded representative for $U$ in a  BSDE solution $(Y,Z,U)$ with bounded $Y$. 

\begin{lemma} \label{basicprops2}
Let $(Y,Z,U)\in \mathcal{S}^{\infty} \times \mathcal{L}^2(B)\times \mathcal{L}^2(\widetilde{\mu})$ be a solution of some JBSDE $(\ref{BSDE})$ with data $(\xi, f)$.
 Then there exists a representative $U'$ of $U$, bounded pointwise by $2|Y|_{\infty}$, such that $U'=U$ in $\mathcal{L}^2(\widetilde{\mu})$ and $\mathbb{P}\otimes {\rm d}t$-a-e., and $(Y,Z,U')$ solves the BSDE $(\xi ,f)$.
 \end{lemma}
\begin{proof}
We reproduce a brief argument sufficient to our general setting, similarly to e.g.\  \cite[Cor.1]{Morlais09}  or \cite[proof of Thm.3.5]{Becherer06}.
Use that $\mu(\omega,{\rm d}t,{\rm d}e) = \sum_{s\ge 0} \mathds{1}_{D}(\omega,s)\,\updelta_{(s,\beta_s(\omega))}({\rm d}t,{\rm d}e)$  for an optional $E$-valued process $\beta$ and a thin set $D$, since
 $\mu$ is an integer-valued random measure \cite[II.\S1b]{JacodShiryaev03}. Clearly the jump
$\Delta Y_t(\omega)=(Y_t-Y_{t-})(\omega)=\int_E U_t(\omega,e)\,\mu(\omega;\{t\},{\rm d}e)$ is  equal to
$\mathds{1}_D(\omega,t)U_t(\omega ,\beta_t(\omega ))$  and  bounded
by $2|Y|_{\infty}$. For $U'_t(\omega ,e):=U_t(\omega ,e)\mathds{1}_D(\omega ,t)\mathds{1}_{\{\beta_t\}}(e)$, we have $U_t(\omega ,\beta_t(\omega ))=U'_t(\omega ,\beta_t(\omega ))$ on $D$, and $\nolinebreak{\sum_{s\ge 0} \mathds{1}_D(\omega,s)|U_s-U'_s|^2(\omega ,\beta_s(\omega))=0}$ 
implies $E[  |U-U'|^2*\nu_T]= E[|U-U'|^2*\mu_T]=0$. Since $U=U'$ in $\mathcal{L}^2(\widetilde{\mu})$ and $U_t=U'_t$ in $L^0(\mathcal{B}(E),\lambda )$, the BSDE is solved by $(Y,Z,U')$.
\end{proof}
Under these conditions, we can and will take $U$ to be bounded by twice the norm of $Y$; Defining $|U|_{\infty}:=\esssup_{(\omega,t,e)} |U_t(e)|$ for $U\in \mathcal{L}^2(\widetilde{\mu})$ yields $|U|_{\infty}\leq 2|Y|_{\infty}$ for any bounded BSDE solution $(Y,Z,U)$.
The next lemma  notes that the stochastic integrals of bounded JBSDE solutions are BMO-martingales when some truncated generator function is bounded from above (below) by $+(-)\langle M\rangle$ for a BMO-martingale $M$; Moreover, their BMO-norms depend only on $|Y|_{\infty}$, the BMO-norm of $M$ and the horizon $T$. 
See \cite[Lem.1.3]{KentiaPhD} for details of the proof, 
and note that BMO-properties of  integrals of (bounded) BSDEs are of course a well-studied topic, cf.\ \cite{ManiaChik14} and references therein. 

\begin{lemma}   \label{BMOproperty}
Let $(Y,Z,U)\in\set$ be a bounded solution to the BSDE $(\xi ,f)$.
Assume there is $M\in{\rm BMO}(\PP )$ such that $\nolinebreak{\int_t^T f_s(Y_{s-},Z_s,U_s)\, {\rm d}s\leq \langle M\rangle_T-\langle M\rangle_t}$ or $-\int_t^T f_s(Y_{s-},Z_s,U_s)\, {\rm d}s$ $\leq \langle M\rangle_T-\langle M\rangle_t$.
Then $\int Z\, {\rm d}B$ and $U*\widetilde{\mu}$ are {\rm BMO}-martingales and their {\rm BMO}-norms (resp.\ $L^2$-norms) are bounded by a constant depending on $|Y|_{\infty}$ and $\lVert M\rVert_{{\rm BMO}(\PP )}$ (resp. on $|Y|_{\infty}$, $\lVert M\rVert_{\mathcal{M}^2}$).
\end{lemma}

\section{Comparison theorems and a-priori-estimates}\label{sec:comp}

The stage for the main comparison Theorem~\ref{comparetheo} and the a-priori-$L^\infty$-estimate of Proposition~\ref{estimatetheo} in this section
is set by the next proposition.
Its line of proof follows the seminal Theorem~2.5 by\cite{Royer06}, with slight generalizations that are needed in the sequel. 
Just some details for the change of measure argument are elaborated a bit differently, 
measurable dependencies of the random field $\gamma$ are specified in more detail, 
and less is assumed on the generators. Instead of imposing specific conditions on the generators which imply existence of solutions, we only insist that we have solutions and impose a generalized $({\rm \bf A}_{\bm \gamma})$-condition as explained in Example~\ref{exagamma}.1.

\begin{proposition}  \label{comparegeneral} 
Let $(Y^i,Z^i,U^i)\in \mathcal{S}^{\infty}\times \mathcal{L}^2(B)\times \mathcal{L}^2(\widetilde{\mu})$ be solutions to the BSDE $(\ref{BSDE})$
for data  $(\xi_i, f_i)$,  $i=1,2$. Assume that  $f_2$ is 
Lipschitz continuous  w.r.t.\  $y$ and $z$.
Let $\gamma :\Omega \times [0,T]\times \RR^{d+3} \times E\rightarrow [-1,\infty )$ with $(\omega ,t,y,z,u,u',e)\mapsto \gamma_t^{y,z,u,u'}(e)$
be a $\mathcal{P}\otimes \BB (\RR^{d+3})\otimes \BB (E)$-measurable function such that for $\overline{\gamma} :=\gamma^{Y_{-}^2,Z^2,U^1,U^2}$ it holds
 \begin{equation}\label{gamma**}
 \begin{split}
&f_2(t,Y_{t-}^2,Z_t^2,U_t^1) - f_2(t,Y_{t-}^2,Z_t^2,U_t^2) \leq \int_E \overline{\gamma}_t(e)\, (U_t^1(e)-U_t^2(e))\, \zeta (t,e)\, \lambda ({\rm d}e),\ \PP\otimes {\rm d}t\text{-a.e.}\\
&\mbox{and the stochastic exponential $\mathcal{E}(\int \beta\, {\rm d}B+\overline{\gamma}*\widetilde{\mu})$ is a martingale  for $\beta$ from $(\ref{linearization})$.}
\end{split}
\end{equation}
 Then a comparison result holds, that means that the inequalities
$\xi_1\leq \xi_2$ and $f_1(t,Y_{t-}^1,Z_t^1,U_t^1)\leq f_2(t,Y_{t-}^1,Z_t^1,U_t^1)$, $\PP\otimes {\rm d}t\text{-a.e.}$, together imply $Y_t^1\leq Y_t^2$ for all $t\leq T$.
\end{proposition}

In results like the above, in \cite{Royer06} and further enchancements \cite{QuenezSulem13,KrusePopier16,Yao17}, the key assumption needed for comparison 
is the existence of an abstract random field $\gamma$ such that inequalities are satisfied between processes. 
In contrast, the subsequent results of this section offer sufficient criteria for comparison that can be verified more easily
 by checking concrete dependencies  w.r.t.\ to basically Euclidean arguments for generator functions $f$ of the type (\ref{generator1}). See also \cite{GeissSteinicke17} for a simpler version in a setting with  a jump measure of  L\'evy-type on $E=\RR^1\setminus\{0\}$ and  $\zeta\equiv1$ .

\begin{proof}
We define  $\widehat{\xi}:=\xi_1-\xi_2$, $\widehat{Y}:=Y^1-Y^2$, $\widehat{Z}:=Z^1-Z^2$ and $\widehat{U}:=U^1-U^2$.
The processes
\begin{align} \label{linearization}
\alpha_s & := \mathds{1}_{\{Y_{s-}^1\neq Y_{s-}^2\}} \frac{f_2(s,Y_{s-}^1,Z_s^1,U_s^1)-f_2(s,Y_{s-}^2,Z_s^1,U_s^1)}{(Y_{s-}^1-Y_{s-}^2)}\,, \nonumber \\
\beta_s  &:=  \mathds{1}_{\{  Z_s^1\neq Z_s^2   \}} \frac{f_2(s,Y_{s-}^2,Z_s^1,U_s^1)-f_2(s,Y_{s-}^2,Z_s^2,U_s^1)}{\lVert Z_s^1-Z_s^2 \rVert^2} (Z_s^1-Z_s^2)
\end{align}
and $R_t :=\exp (\int_0^t \alpha_s\, {\rm d}s)$ are bounded due to the Lipschitz assumption on $f_2$.
As in \cite{Royer06}, applying It\^o's formula to $R\widehat{Y}$ between $\tau\wedge t$ and $\tau\wedge T$ for some stopping times $\tau$ yields
\begin{align*}
(R\widehat{Y})_{\tau\wedge t} &= (R\widehat{Y})_{\tau\wedge T}+\int_{\tau\wedge t}^{\tau\wedge T} R_s\big(f_1(s,Y_{s-}^1,Z_s^1,U_s^1)-f_2(s,Y_{s-}^2,Z_s^2,U_s^2)\big)\, {\rm d}s\\
&\quad -\int_{\tau\wedge t}^{\tau\wedge T} R_s\widehat{Z}_s\, {\rm d}B_s-\int_{\tau\wedge t}^{\tau\wedge T} \int_E R_s\widehat{U}_s(e)\, \widetilde{\mu}({\rm d}s,{\rm d}e)-\int_{\tau\wedge t}^{\tau\wedge T} R_s\alpha_s\widehat{Y}_{s-}\, {\rm d}s.
\end{align*}
Set $M:=\int R\widehat{Z}\, {\rm d}B+(R\widehat{U})*\widetilde{\mu}$ and $N:=\int \beta\, {\rm d}B+\overline{\gamma} *\widetilde{\mu}$.
Then ${\rm d}\QQ :=\EEE (N)_T{\rm d}\PP$ defines an absolutely continuous probability by the martingale property of the stochastic exponential $\mathcal{E}(N)\ge 0$;
cf.\ \cite[Lem.9.40]{HeWangYan92}. By Girsanov $L:=M-\langle M,N\rangle$ is a local $\QQ$-martingale,
and the inequality
\begin{align}\nonumber
\hspace{-0.2cm}f_1(s,Y_{s-}^1,Z_s^1,U_s^1)-f_2(s,Y_{s-}^2,Z_s^2,U_s^2) &\leq \alpha_s \widehat{Y}_{s-}+\beta_s \widehat{Z}_s+\int_E \overline{\gamma}_s(e)\widehat{U}_s(e)\,\zeta_s (e)\,\lambda ({\rm d}e)\, \PP \otimes {\rm d}s\text{-a.e.}\\
\text{implies} \quad   \label{esti}
(R\widehat{Y})_{\tau\wedge t}&\leq (R\widehat{Y})_{\tau\wedge T}-(L_T^{\tau}-L_t^{\tau}).
\end{align}
Localizing $L$ along a sequence of stopping times $\tau_n\uparrow \infty$ and taking conditional expectations, we obtain $\nolinebreak{\EE_{\QQ} \big((R\widehat{Y})_{t\wedge \tau^n}\,\big\lvert\,\FF_t\big)\leq \EE_{\QQ} \big((R\widehat{Y})_{\tau^n\wedge T}\,\big\lvert\,\FF_t \big)}$ for each $n\in \NN$.
Dominated convergence  yields the estimate $\nolinebreak{R_t\widehat{Y}_t\leq \EE_{\QQ} \big(R_T\widehat{\xi}\,\big\lvert\,\FF_t\big)\leq 0}$ and thus $Y_t^1\leq Y_t^2$.
\end{proof}

\begin{remark}
\begin{enumerate}\item[1.] Switching roles of $f_1$ and $f_2$, one gets that if $f_1$ is Lipschitz in $y$,$z$ and satisfies $(\ref{gamma**})$ instead of $f_2$, then 
$\xi_1\leq \xi_2$ and $f_1(t,Y_{t-}^2,Z_t^2,U_t^2)\leq f_2(t,Y_{t-}^2,Z_t^2,U_t^2)$ imply $Y_t^1\leq Y_t^2$.
\item[2.] The result of Proposition \ref{comparegeneral} remains valid (with a similar proof) if one requires that the $Y$-components of JBSDE solutions to compare are in $\mathcal{S}^2$
instead of $\mathcal{S}^\infty$, and the stochastic exponential $\mathcal{E}(\beta\mal B+\overline{\gamma}\ast\widetilde{\mu})$ is in $\mathcal{S}^2$. However, as it is stated, 
Proposition \ref{comparegeneral} is exactly what we will need to apply in the sequel to derive, e.g., Proposition~$\ref{finitetheo}$ and Theorem~$\ref{infinitetheo}$.
\end{enumerate}
\end{remark}

\begin{example} \label{exmartingale}
Sufficient conditions for $\EEE (\overline{\gamma}*\widetilde{\mu})$ to be a martingale are, for instance,
\begin{enumerate}
\item \label{exmartingale1} $\Delta (\overline{\gamma}*\widetilde{\mu})>-1$ and $\EE \big(\exp (\langle \overline{\gamma}*\widetilde{\mu} \rangle_T )\big) =\EE \big( \exp \big( \int_0^T\int_E |\overline{\gamma}_s(e)|^2\, \nu ({\rm d}s,{\rm d}e)\big) \big)<\infty$;  see~\cite[Thm.9]{ProtterShimbo08}.
This holds i.p.\ if $\int_E |\overline{\gamma}_s(e)|^2\, \zeta (s,e)\, \lambda ({\rm d}e)<const.<\infty$ $\PP \otimes {\rm d}s$-a.e.\  and $\overline{\gamma}>-1$.
\item \label{exmartingale2} $\Delta (\overline{\gamma}*\widetilde{\mu})\geq -1+\delta$ for  $\delta >0$ and $\overline{\gamma}*\widetilde{\mu}$ is a ${\rm BMO}(\PP )$-martingale due to Kazamaki \cite{Kazamaki79}.
\item \label{exmartingale3}  $\Delta (\overline{\gamma}*\widetilde{\mu})\geq -1$ and $\overline{\gamma}*\widetilde{\mu}$ is a uniformly integrable martingale and $\EE \big(\exp (\langle \overline{\gamma}*\widetilde{\mu} \rangle_T )\big)<\infty$; see~\cite[Thm.I.8]{LepingleMemin78}.
Such a condition is satisfied when $\overline{\gamma}$ is bounded and $\lvert \overline{\gamma}\rvert\le \psi,\ \PP\otimes {\rm d}t\otimes \lambda\text{-a.e.}$ for a function $\psi\in L^2(\lambda)$ and $\zeta\equiv 1$. The latter is what is 
required for instance in the comparison Thm.4.2 of \cite{QuenezSulem13}.
\end{enumerate}
Note that under above conditions, also the stochastic exponential $\EEE (\int \beta {\rm d}B+\overline{\gamma}*\widetilde{\mu})$
 for $\beta$  bounded and predictable  is a martingale, as it is easily seen by Novikov's criterion. 

 Let us also  refer to  \cite[Sections 19 and A.9]{CohenElliot10book} for related so-called balance conditions on generators for JBSDE comparison by change of measure arguments.
 \end{example}
In the statement of Proposition \ref{comparegeneral}, the dependence of the process $\overline{\gamma}$ on the BSDE solutions is not needed for the proof as the same result holds if $\overline{\gamma}$
is just a predictable process such that the estimate on the generator $f_2$ and the martingale property (\ref{gamma**}) hold. The further functional dependence is needed for the sequel, as 
required in the following
\begin{definition} \label{agammadef}
We say that an $\overline{\RR}$-valued generator function $f$ satisfies condition ${\bf ({\rm \bf A}_{\bm \gamma})}$ if there is a $\mathcal{P}\otimes \BB (\RR^{d+3})\otimes \BB (E)$-measurable function $\gamma :\Omega \times [0,T]\times \RR^{d+3} \times E\rightarrow (-1,\infty )$ given by $(\omega ,t,y,z,u,u',e)\mapsto \gamma_t^{y,z,u,u'}(e)$ such that for all $(Y,Z,U,U')\in \mathcal{S}^{\infty}\times \mathcal{L}^2(B)\times (\mathcal{L}^2(\widetilde{\mu}))^2$ with $|U|_{\infty}< \infty$, $|U'|_{\infty}<\infty$ it holds for $\overline{\gamma}:=\gamma^{Y_-,Z,U,U'}$
\begin{equation} \label{gamma*}
\begin{split}
&f_t(Y_{t-},Z_t,U_t)\! -\! f_t(Y_{t-},Z_t,U'_t)\! \leq\! \int_E\! \overline{\gamma}_t(e)(U_t(e)-U'_t(e)) \zeta (t,e) \lambda ({\rm d}e),\, \PP\otimes {\rm d}t\text{-a.e.}\\ 
&\mbox{and $\mathcal{E}(\int \beta {\rm d}B+\overline{\gamma}*\widetilde{\mu})$ is a martingale for every bounded and predictable $\beta$.}
\end{split}
\end{equation}
We will say that $f$ satisfies condition $({\rm \bf A}'_{\bm \gamma})$ if the above holds for all bounded $U$ and $U'$ with additionally $U*\widetilde{\mu}$ and $U'*\widetilde{\mu}$ in ${\rm BMO}(\PP )$.
\end{definition}
Clearly, existence and applicability of a suitable comparison result for solutions to JBSDEs implies their uniqueness. In other words, if there exists a bounded solution for a
generator being Lipschitz w.r.t.\ $y$ and $z$ which satisfies $({\rm \bf A}_{\bm \gamma})$ or $({\rm \bf A}'_{\bm \gamma})$, we obtain that such a solution is unique.
\begin{example}  \label{exgamma}
The natural candidate for $\gamma$ for generators $f$ of the form $(\ref{generator1})$ is given by 
\begin{align}
\gamma_s^{y,z,u,u'}(e)=\frac{g_s(y,z,u,e)-g_s(y,z,u',e)}{u-u'}\ \mathds{1}_A(e)\,\mathds{1}_{\{u\neq u'\}},  \label{gammaform}
\end{align}
which is $\mathcal{P}\otimes \BB (\RR^{d+3} )\otimes \BB (E)$-measurable since $g$ is.
Assuming  absolute continuity of $g$ in $u$, we can express $\gamma_s^{y,z,u,u'}(e)= \int_0^1 \frac{\partialup}{\partialup u}g_s(y,z,tu+(1-t)u',e)\, {\rm d}t\, \mathds{1}_A(e)$, by noting that
\begin{equation*}
(u-u')\int_0^1 \frac{\partialup}{{\partialup} u}g_s(y,z,tu+(1-t)u',e)\, {\rm d}t\, \mathds{1}_A(e) = \int_0^1 \frac{\partialup}{\partialup t}\left[(g_s(y,z,tu+(1-t)u',e))\right]\, {\rm d}t\, \mathds{1}_A(e).
\end{equation*}
For generators of type $(\ref{generator})$ the $\gamma$ simply is
\( \gamma_s^{y,z,u,u'}(e)=\int_0^1 \frac{\partialup}{\partialup u}g_s(tu+(1-t)u',e)\, {\rm d}t\, \mathds{1}_A(e)\,.
\)
\end{example}

\begin{definition} \label{defAfininfi}
We say that a generator $f$ satisfies condition $({\rm \bf A}_{\rm \bf fin})$ or $({\rm \bf A}_{\rm \bf infi})$   (on a set $D$) if
\begin{enumerate}
\item $({\rm \bf A}_{\rm \bf fin})$: $f$ is of the form $(\ref{generator1})$ with $\lambda (A)<\infty$, is Lipschitz continuous w.r.t.\ $y$ and $z$ uniformly in $(t,\omega,u)$, and the map $u \mapsto g(t,y,z,u,e)$ is 
absolutely continuous (in $u$) for all $(\omega ,t,y,z,e)$ (in $D\subseteq \Omega \times [0,T]\times \RR \times \RR^d \times E$), i.e.\ $g(t,y,z,u,e) = g(0)+\int_0^ug'(t,y,z,x,e){\rm d}x$, with  density function $g'$ being strictly greater than $-1$ (on $D$)
and locally bounded (in u) from above, uniformly in $(\omega ,t,y,z,e)$.

\item $({\rm \bf A}_{\rm \bf infi})$: $f$ is of the form $(\ref{generator})$, is Lipschitz continuous w.r.t.\ $y$ and $z$ uniformly in $(t,\omega,u)$, and the map $u \mapsto g_t(u,e)$ is 
absolutely continuous (in $u$) for all $(\omega ,t,e)$ (in $D$), i.e.\ $g(t,u,e) = g(0)+\int_0^ug'(t,x,e){\rm d}x$, with  density function $g'$ being such that 
for all $c\in(0,\infty)$ there exists $K(c)\in\mathbb{R}$ and $\delta(c)\in(0,1)$ with 
$-1+\delta(c)\le g'(x)$ and $\lvert g'(x)\rvert\le K(c)\lvert x\rvert$ for all $x$ with $\lvert x\rvert \le c.$
\end{enumerate}
\end{definition}
\begin{remark}
Note that under condition $({\rm \bf A}_{\rm \bf infi})$ the density function $g'$ is necessarily locally bounded, in particular with $\lvert g'(x)\rvert\le K(c)c=:\bar{K}(c)<\infty$ for all $x\in[-c,c]$. Observe that the conditions are not requiring the function $g$ to be  convex and moreover refrain from requiring it to be continuously differentiable in $u$. Both can be helpful in application examplres, see 
Section~\ref{subsubsec:CaseDiscAssetPrice}.
 
\end{remark}

\begin{example} Sufficient conditions for condition $({\rm \bf A}_{\bm \gamma})$ and $({\rm \bf A}'_{\bm \gamma})$ are \label{exagamma}
\begin{enumerate}\item[1.] $\gamma$ is a $\mathcal{P}\otimes \mathcal{B}(\RR^{d+3})\otimes \mathcal{B}(E)$-measurable function satisfying the inequality in $(\ref{gamma*})$ and
\[
C_1(1\wedge |e|)\leq \gamma_t^{y,z,u,u'}(e)\leq C_2(1\wedge |e|)
\]
on $E=\RR^l\setminus\{ 0\}$ ($l\in\NN$), for some $C_1\in (-1,0]$ and $C_2>0$. In this case $\exp (\langle \int \beta {\rm d}B+\overline{\gamma} *\widetilde{\mu} \rangle_T )$ is clearly
bounded and the jumps of $\int \beta {\rm d}B+\overline{\gamma} *\widetilde{\mu}$ are bigger than $-1$. Hence $\mathcal{E}\left(\int \beta {\rm d}B+\overline{\gamma} *\widetilde{\mu}\right)$
is a positive martingale \cite[Thm.9]{ProtterShimbo08}. Thus Definition~\ref{agammadef} generalizes the original $({\rm \bf A}_{\bm \gamma})$-condition introduced by \cite{Royer06} for Poisson random measures.
\item[2.] $({\rm \bf A}_{\rm \bf fin})$ is sufficient for $({\rm \bf A}_{\bm \gamma})$. This follows from Example~\ref{exmartingale}.\ref{exmartingale1}, $(\ref{gammaform})$ and $\lambda (A)<\infty$.
\item[3.] $({\rm \bf A}_{\rm \bf infi})$ is sufficient for $({\rm \bf A}'_{\bm \gamma})$. To see this, let $u,u'$ be bounded by $c$ and $\gamma$ be the natural candidate in Example \ref{exgamma}. 
Then $|\gamma_s^{y,z,u,u'}(e)|\leq \int_{u'}^u|g'(x)|{\rm d}x/(u-u')\le K(c)(|u|+|u'|)$. Hence $\int \beta {\rm d}B+\overline{\gamma}*\widetilde{\mu}$ is a BMO-martingale by the BMO-property of $U*\widetilde{\mu}$ and $U'*\widetilde{\mu}$ with 
some lower bound $-1+\delta$ for its jumps. And $\mathcal{E}(\int \beta {\rm d}B+\overline{\gamma}*\widetilde{\mu})$ is a martingale by 
 part~\ref{exmartingale2} of Example~$\ref{exmartingale}$.
\item[4.] Condition $({\rm \bf A}_{\rm \bf fin})$ above is satisfied if, e.g., $f$ is of the form $(\ref{generator1})$ with $\lambda (A)<\infty$, is Lipschitz continuous w.r.t.\ $y$ and $z$, and the map $u \mapsto g(t,y,z,u,e)$ is continuously
differentiable for all $(\omega ,t,y,z,e)$ (in $D$) such that the derivative is strictly greater than $-1$ (on $D\subseteq \Omega \times [0,T]\times \RR \times \RR^d \times E$)
and locally bounded (in $u$) from above, uniformly in $(\omega ,t,y,z,e)$.
\item[5.] Condition $({\rm \bf A}_{\rm \bf infi})$ is valid if for instance $f$ is of the form $(\ref{generator})$, is Lipschitz continuous w.r.t.\ $y$ and $z$, and the map $u \mapsto g_t(u,e)$ is twice continuously differentiable for all
$(\omega ,t,e)$ with the derivatives being locally bounded uniformly in $(\omega ,t,e)$, the first derivative being 
(locally)
bounded away from $-1$ with a lower bound $-1+\delta$ for some $\delta >0$,
and $\frac{\partialup g}{\partialup u}(t,0,e)\equiv 0$.
\end{enumerate}
\end{example}

As an application of the above, we can now provide simple  conditions for comparison in terms of concrete properties of the generator function, which are easier to verify than the more general but abstract conditions on the existence of a suitable function $\gamma$ as in Proposition~\ref{comparegeneral} or the general conditions by \cite{CohenElliott10}.
Note that no convexity is required in the $z$ or $u$ argument of the generator. The result will be applied later to prove existence and uniqueness of JBSDE solutions. 

\begin{theorem}[Comparison Theorem] \label{comparetheo}
A comparison result between bounded BSDE solutions in the sense of Proposition~\ref{comparegeneral} holds true in each of the following cases:
\begin{enumerate}
\item \emph{(finite activity)} $f_2$ satisfies $({\rm \bf A}_{\rm \bf fin})$.
\item \emph{(infinite activity)} $f_2$ satisfies $({\rm \bf A}_{\rm \bf infi})$ and $U^1*\widetilde{\mu}$ and $U^2*\widetilde{\mu}$ are {\rm BMO}($\PP$)-martingales for the corresponding JBSDE solutions $(Y^1,Z^1,U^1)$ and $(Y^2,Z^2,U^2)$.
\end{enumerate}
\end{theorem}
\begin{proof}
This follows directly from Proposition~\ref{comparegeneral} and Example~\ref{exagamma}, noting that representation $(\ref{gammaform})$ in connection with condition $({\rm \bf A}_{\rm \bf fin})$ resp. $({\rm \bf A}_{\rm \bf infi})$ meets the sufficient conditions in Example~\ref{exmartingale}.
\end{proof}

Unlike classical a-priori estimates that offer some $L^2$-norm estimates for the BSDE solution in terms of the data, the next result gives a simple $L^{\infty}$-estimate
for the $Y$-component of the solution. Such will be useful for the derivation of BSDE solution bounds and for truncation arguments.
\begin{proposition}\label{estimategeneral} 
Let $(Y,Z,U)\in \set$ be a solution to the BSDE $(\xi ,f)$ with $\xi \in L^{\infty}(\FF_T)$, $f$ be Lipschitz continuous w.r.t.\ $(y,z)$ with Lipschitz constant $K_f^{y,z}$ and
satisfying $({\rm \bf A}_{\bm \gamma})$ with $f_.(0,0,0)$ bounded. Then $|Y_t|\leq \exp \big(K_f^{y,z}(T-t)\big)\big(|\xi |_{\infty}+(T-t)|f_.(0,0,0)|_{\infty}\big)$ for $t\leq T$.
\end{proposition}

\begin{proof}
Set $(Y^1,Z^1,U^1)=(Y,Z,U)$, $(\xi^1,f^1)=(\xi ,f)$, $(Y^2,Z^2,U^2)=(0,0,0)$ and $(\xi^2,f^2)=(0,f)$. Then following the proof of Proposition \ref{comparegeneral}, equation ($\ref{esti}$) becomes
\[
(RY)_{\tau\wedge t} \leq (RY)_{\tau\wedge T}+\int_{\tau\wedge t}^{\tau\wedge T} R_sf_s(0,0,0)\, {\rm d}s-(L_T^{\tau}-L_t^{\tau}), \quad t\in [0,T],
\]
for all stopping times $\tau $ where $L:=M-\langle M,N\rangle$ is in $\MM_{\rm loc}(\QQ )$, $M:=\int RZ\, {\rm d}B+(RU)*\widetilde{\mu}$ is in $\mathcal{M}^2$, $N:=\int \beta\, {\rm d}B+\overline{\gamma}*\widetilde{\mu}$ with 
$\overline{\gamma}:=\gamma^{0,0,U,0}$ and the probability measure $\QQ\approx \PP$ is given by ${\rm d}\QQ :=\EEE (N)_T{\rm d}\PP$. Localizing $L$ along some sequence $\tau^n\uparrow \infty$ of stopping times yields $\EE_{\QQ}\big((RY)_{\tau^n\wedge t}\,\big\lvert\,\FF_t\big )\leq \EE_{\QQ}\big((RY)_{\tau^n\wedge T}+\int_{\tau\wedge t}^{\tau\wedge T} R_sf_s(0,0,0)\, {\rm d}s\,\big\lvert\,\FF_t\big)$.
By dominated convergence, we conclude that $\PP$-a.e
\[
Y_t \leq \EE_{\QQ}\Big( \frac{R_T}{R_t}\xi +\int_t^T \frac{R_s}{R_t}f_s(0,0,0)\, {\rm d}s\, \Big|\, \FF_t\Big) \leq {\rm e}^{K_f^{y,z}(T-t)}\big(|\xi |_{\infty}+(T-t)|f_\cdot(0,0,0)|_{\infty}\big).
\]
Analogously, if we define $\overline{N}:=\int \beta\, {\rm d}B+\overline{\widetilde{\gamma}}*\widetilde{\mu}$ with $\overline{\widetilde{\gamma}}:=\gamma^{0,0,0,U}$, and $\overline{\QQ}$ equivalent to $\PP$ via ${\rm d}\overline{\QQ}:=\EEE (\overline{N})_T{\rm d}\PP$, we deduce that $\overline{L}:=M-\langle M,\overline{N}\rangle$ is in $\MM_{\rm loc}(\overline{\QQ})$ and
\[
(RY)_{\tau\wedge t} \geq (RY)_{\tau\wedge T}+\int_{\tau\wedge t}^{\tau\wedge T} R_sf_s(0,0,0)\, {\rm d}s-(\overline{L}_T^{\tau}-\overline{L}_t^{\tau}), \quad t\in [0,T],
\]
for all stopping times $\tau$. This yields the required lower bound.
\end{proof}

Again, we can specify explicit conditions on the generator function that are sufficient to ensure the more abstract assumptions of the previous result.
\begin{proposition} \label{estimatetheo}
Let $(Y,Z,U)\in \set$ be a solution to the BSDE $(\xi ,f)$ with $\xi$ in $L^{\infty}(\FF_T)$, $f$ being Lipschitz continuous w.r.t.\ $(y,z)$ with Lipschitz constant $K_f^{y,z}$
such that $f_.(0,0,0)$ is bounded. Assume that one of the following conditions holds:
\begin{enumerate}
\item \emph{(finite activity)} $f$ satisfies $({\rm \bf A}_{\rm \bf fin})$.
\item \emph{(infinite activity)} $f$ satisfies $({\rm \bf A}_{\rm \bf infi})$ and $U*\widetilde{\mu}$ is a {\rm BMO}($\PP$)-martingale.
\end{enumerate}
Then $|Y_t|\leq \exp \big(K_f^{y,z}(T-t)\big)\big(|\xi |_{\infty}+(T-t)|f_s(0,0,0)|_{\infty}\big)$ holds for all $t\leq T$, in particular $|Y|_{\infty}\leq \exp \big(K_f^{y,z}T\big)\big(|\xi |_{\infty}+T|f_s(0,0,0)|_{\infty}\big)$.
\end{proposition}

\begin{proof}
This follows directly from Proposition~\ref{estimategeneral} and Example~\ref{exagamma},  since $f$ satisfies condition $({\rm \bf A}_{\bm \gamma})$ (resp. $({\rm \bf A}'_{\bm \gamma})$) using equation~(\ref{gammaform}).
\end{proof}

In the last part of this section we apply our comparison theorem for more concrete generators. To this end, we consider 
 a generator $f$  being truncated at  bounds
$a<b$ (depending on time only) as
\begin{align} \label{tildef}
\widetilde{f}_t(y,z,u):=f_t\big(\kappa (t,y),\,z,\,\kappa (t,y+u)-\kappa (t,y)\big),
\end{align}
with $\kappa (t,y):=\big(a(t)\vee y\big)\wedge b(t)$. Next, we show that if a generator satisfies $({\rm \bf A}_{\bm \gamma})$ within the truncation bounds, then the truncated generator satisfies $({\rm \bf A}_{\bm \gamma})$ everywhere.
\begin{lemma} \label{trunclemma}
Let $f$ satisfy~(\ref{gamma*}) for $Y,U$ such that $a(t)\leq Y_{t-},Y_{t-}+U_t(e),Y_{t-}+U'_t(e)\leq b(t)$, $t\in [0,T]$ and let $\gamma$ satisfy one of the conditions of
Example~\ref{exmartingale} for the martingale property of $\mathcal{E}(\overline{\gamma}*\widetilde{\mu})$. Then $\widetilde{f}$ satisfies~(\ref{gamma*}). Especially, if $f$ satisfies
$({\rm \bf A}_{\rm \bf fin})$ on the set where $a(t)\leq y,y+u \leq b(t)$ then $\widetilde{f}$ is Lipschitz in $(y,z)$, locally Lipschitz in $u$ and satisfies $({\rm \bf A}_{\bm \gamma})$.
\end{lemma}
\begin{proof}
Using  monotonicity of $x\mapsto \kappa (t,x)$, we get that $\widetilde{f}_t(Y_{t-},Z_t,U_t)-\widetilde{f}_t(Y_{t-},Z_t,U'_t)$  equals
\begin{align*}
& f_t\big(\kappa (t,Y_{t-}),Z_t,\kappa (t,Y_{t-}+U_t)-\kappa (t,Y_{t-})\big)-f_t\big(\kappa (t,Y_{t-}),Z_t,\kappa (t,Y_{t-}+U'_t)-\kappa (t,Y_{t-})\big) \\
& \quad \leq \int_E \overline{\gamma}_t(e) \big(\kappa (t,Y_{t-}+U_t(e)) -\kappa (t,Y_{t-}+U'_t(e))\big)\, \zeta (t,e)\, \lambda ({\rm d}e) \\
& \quad \leq \int_E \overline{\gamma}_t(e) \big(\mathds{1}_{\{ \overline{\gamma}\geq 0, U\geq U'\} }+\mathds{1}_{\{ \overline{\gamma}<0, U<U'\}}\big)\big(U_t(e)-U'_t(e)\big)\, \zeta (t,e)\, \lambda ({\rm d}e).
\end{align*}
Setting $\overline{\gamma^*}:= \overline{\gamma}\big(\mathds{1}_{\{ \overline{\gamma}\geq 0, U\geq U'\} }+\mathds{1}_{\{ \overline{\gamma}<0, U<U'\} }\big)$ we see that the stochastic exponential 
$\mathcal{E}\big(\int \beta {\rm d}B+\overline{\gamma^*}*\widetilde{\mu}\big)$ is a martingale for all bounded and predictable processes $\beta$ and $\widetilde{f}$ satisfies~(\ref{gamma*}).
The latter claim easily follows from the fact that if $f$ satisfies $({\rm \bf A}_{\rm \bf fin})$ on $a(t)\leq y,y+u \leq b(t)$ then $f$ satisfies~(\ref{gamma*}) on $a(t)\leq Y_{t-},Y_{t-}+U_t(e),Y_{t-}+U'_t(e) \leq b(t)$
using Example~\ref{exagamma}.2. The Lipschitz properties of $\widetilde{f}$ follow from the fact that $\kappa$ is a contraction and $f$ is Lipschitz within the truncation bounds.
\end{proof}

Concrete $L^\infty$-bounds for bounded solutions to BSDE $(\xi ,f)$ with suitable $\widehat{f}$-part are provided by
\begin{proposition} \label{ODEbounds}
Let $f$ be a generator of the form $(\ref{generator1})$ with $\big|\widehat{f}_t(y,z)\big|\leq K_1+K_2|y|$ for some $K_1,K_2\geq 0$, $g_t(y,z,0,e)\equiv 0$ and $\xi \in L^{\infty}(\FF_T)$ with
$c_1\leq \xi \leq c_2$ for some $c_1,c_2\in \RR$. Assume that there are solutions $a$ and $b$ to the ODEs $y'(t)=K_1+K_2|y(t)|$, $y(T)=c_1$ and $y'(t)=-(K_1+K_2|y(t)|)$,
$y(T)=c_2$ respectively, such that $a\leq b$ on $[0,T]$. If the truncated generator $\widetilde{f}$ in $(\ref{tildef})$ satisfies $({\rm \bf A}_{\bm \gamma})$ and is Lipschitz in $(y,z)$, then any solution
$(\widetilde{Y},\widetilde{Z},\widetilde{U})\in\set$ to the JBSDE $(\xi ,\widetilde{f})$ also solves the JBSDE $(\xi ,f)$ and satisfies $a(t)\leq \widetilde{Y}_t\leq b(t),$ $t\in [0,T]$.
\end{proposition}

\begin{proof}
We set $Y_t:=\kappa (t,\widetilde{Y}_t)$, $Z_t:=\widetilde{Z}_t$, $U_t(e):=\kappa (t,\widetilde{Y}_{t-}+\widetilde{U}_t(e))-\kappa (t,\widetilde{Y}_{t-})$ and
\[
f_t^{\,i}(y,z,u):= \widehat{f}_t^{\,i}\big(\kappa (t,y),z\big)+\int_E g_t\big(\kappa (t,y),z,\kappa (t,y+u)-\kappa (t,y),e\big)\, \zeta (t,e)\, \lambda ({\rm d}e)
\]
with $\widehat{f}_t^{\,1}(y,z):=-(K_1+K_2|y|)$, $\widehat{f}_t^{\,2}(y,z):=\widehat{f}_t(y,z)$ and $\widehat{f}_t^{\,3}(y,z):=K_1+K_2|y|$. By the assumptions on the ODEs, we have
that $(a(t),0,0)$ solves the BSDE $(c_1,f^1)$ and $(b(t),0,0)$ solves the BSDE $(c_2,f^3)$. Taking into account that $\widetilde{f}^{\,1}\leq \widetilde{f}^{\,2}\leq \widetilde{f}^{\,3}$, $c_1 \leq \xi \leq c_2$ and 
$\widetilde{f}^{\,2}$ satisfies $({\rm \bf A}_{\bm \gamma})$, comparison theorem Proposition~\ref{comparegeneral} yields $a(t)\leq \widetilde{Y}_t\leq b(t)$. Hence, $Y$ and $\widetilde{Y}$ are indistinguishable, $U=\widetilde{U}$ in $\mathcal{L}^2(\widetilde{\mu})$ and $(\widetilde{Y},\widetilde{Z},\widetilde{U})$ solves the BSDE $(\xi ,f)$.
\end{proof}

In the next section, we apply these results to two situations: Using Corollary~\ref{bechereralternative}, we give an alternative proof of Thm.3.5 of \cite{Becherer06}
via a comparison principle instead of an argument with stopping times. Moreover, the estimates in Corollary~\ref{boundscor} are applied to solve the power utility maximization problem via a JBSDE
approach in Section~\ref{subsubsec:powut}.

\section{Existence and uniqueness of bounded solutions}\label{sec:exandu}

This section studies BSDE with jumps by the monotone stability approach. 
Building on (straighforward)  results for finite activity, the infinite activity case is treated by monotone approximations.

\subsection{The case of finite activity}

\begin{definition}
A generator function $f$ satisfies condition $({\rm \bf B}_{\bm \gamma})$, if it is Lipschitz continuous in $(y,z)$, locally Lipschitz continuous in $u$ (in the sense that $u\mapsto f_t(y,z,-c\vee u\wedge c)$ is Lipschitz continuous for any $c\in(0,\infty)$), $f_.(0,0,0)$ is bounded, and $f$ satisfies condition $({\rm \bf A}_{\bm \gamma})$.
\end{definition}
The next result  readily leads to Proposition~\ref{finitetheo}, for $A$  in (\ref{generator1}) with $\lambda(A)<\infty$.
\begin{proposition}  \label{finitetheogeneral}
Let $\xi \in L^{\infty}(\FF_T)$ and $f$ satisfies $({\rm \bf B}_{\bm \gamma})$. Then there exists a unique solution $(Y,Z,U)$ in $\set$ to the BSDE $(\xi ,f)$. Moreover for all $t\in[0,T]$, $|Y_t|$ is bounded
by $\exp \big(K_f^{y,z}(T-t)\big)\big(|\xi |_{\infty}+(T-t)|f_.(0,0,0)|_{\infty}\big)$. \end{proposition}

\begin{proof}
Consider the Lipschitz generator $f_t^c(y,z,u):=f_t\big(y,z,(u\vee (-c))\wedge c\big)$ with $c>0$ and Lipschitz constant $K_{f^c}$. By classical fixed point arguments and a-priori estimates  (cf.\ e.g.\ \cite[Prop.3.2, 3.3]{Becherer06})  there is a unique solution $(Y^c,Z^c,U^c) \in  \mathcal{S}^2\times \mathcal{L}^2(B)\times \mathcal{L}^2(\widetilde{\mu})$ to the BSDE $(\xi ,f^c)$;  it satisfies
\[
|Y_t^c| \leq C \EE \Big( |\xi |^2+\int_t^T |f_s^c(0,0,0)|^2\, {\rm d}s\, \Big|\, \FF_t\Big) \leq C\big(|\xi |_{\infty}^2+T|f_.(0,0,0)|_{\infty}^2\big)<\infty,
\]
for some constant $C=C(T,K_{f^c})$. Now Proposition $\ref{estimategeneral}$ implies that
$
|Y_t^c|
$
is dominated by 
$
\exp \big(K_f^{y,z}(T-t)\big)\big(|\xi |_{\infty}+(T-t)|f_.(0,0,0)|_{\infty}\big)
$ 
{ for all }$c>0$.
Choosing $c\geq 2\exp \big(K_f^{y,z}T\big)\big(|\xi |_{\infty}+T|f_.(0,0,0)|_{\infty}\big)$ we get that $(Y^c,Z^c,U^c)$ with $Y^c\in \mathcal{S}^{\infty}$ solves the BSDE $(\xi ,f)$ since $U^c$ is bounded by $c$. Uniqueness follows by comparison.
\end{proof}

This leads to a preliminary result on bounded solutions if jumps are of finite activity. 
\begin{proposition}   \label{finitetheo}
Let $\xi \in L^{\infty}(\FF_T)$ and let $f$ satisfy $({\rm \bf A}_{\rm \bf fin})$ (recall Definition~\ref{defAfininfi}) with $f_.(0,0,0)$ bounded. Then there exists a unique solution $(Y,Z,U)$ in $\set$ to the BSDE $(\xi ,f)$.
Moreover for all $t\in[0,T]$, $|Y_t|$ is bounded by $\exp \big(K_f^{y,z}(T-t)\big)\big(|\xi |_{\infty}+(T-t)|f_.(0,0,0)|_{\infty}\big)$.  \end{proposition}

\begin{proof}
Noting that local Lipschitz continuity in $u$ follows from the absolute continuity of $g$ in $u$ with locally bounded density function, the claim follows from
Propositions~\ref{estimatetheo} and \ref{finitetheogeneral}.
\end{proof}

\begin{corollary} \label{bechereralternative}
Let $\xi \in L^{\infty}(\FF_T)$ and let $f$ be a generator satisfying $({\rm \bf A}_{\rm \bf fin})$, with $g_t(y,z,0,e)\equiv 0$ and $|\widehat{f}_t(y,z)|\leq K_1+K_2|y|$ for some $K_1,K_2\geq 0$. Set
\begin{equation*}
b(t) = \begin{cases} (|\xi |_{\infty}+\frac{K_1}{K_2})\exp (K_2(T-t))-\frac{K_1}{K_2}, & K_2\neq 0 \\ |\xi |_{\infty}+K_1(T-t), & K_2=0.  \end{cases}
\end{equation*}
Then there exists a unique solution $(Y,Z,U)\in \set$ to the BSDE $(\xi ,f)$ and moreover $|Y_t|\leq b_t$ for $t\in[0,T]$. Finally $\int Z\, {\rm d}B$ and $U*\widetilde{\mu}$ are
{\rm BMO}$(\PP )$-martingales.
\end{corollary}
\begin{proof}
By Lemma~\ref{trunclemma} and Proposition~\ref{finitetheo}, there is a unique solution $(Y,Z,U)$ in the space $\set$ to the BSDE $(\xi ,\widetilde{f})$. By
Proposition~\ref{ODEbounds}, it also solves the BSDE $(\xi ,f)$ and $-b(t)\leq Y_t\leq b(t),\ \forall t\in[0,T]$. Uniqueness follows from the fact that one can apply the comparison Theorem~\ref{comparetheo} for
generators satisfying $({\rm \bf A}_{\rm \bf fin})$. The BMO property follows from Lemma~\ref{BMOproperty}.
\end{proof}

\begin{remark}
Corollary~\ref{bechereralternative} is similar to Thm.3.5 in \cite{Becherer06}, but its proof is different: It relies on previous comparison results for JBSDEs instead of stopping arguments.
  The stochastic integrals of the BSDE solution are BMO-martingales under the assumptions for Lemma~\ref{BMOproperty},
which hold  e.g.\ under the conditions for \cite[Thm.~3.6]{Becherer06}
\end{remark}

\begin{corollary} \label{boundscor}
Let $\xi \in L^{\infty}(\FF_T)$ with $\xi \geq C$ for some constant $C>0$, $K\geq 0$ and set $a(t):=C\exp (-K(T-t))$ and $b(t)=|\xi |_{\infty} \exp (K(T-t)),\ \forall t\in[0,T]$. Assume
$f$ satisfies $({\rm \bf A}_{\rm \bf fin})$ for $c\leq y,y+u\leq d$ for all $c,d\in \RR$ with $0<c<d$, and that $\big|\widehat{f}_t(y,z)\big|\leq K|y|$ and $g_t(y,z,0,e)=0$.
Then there exists a unique solution $(Y,Z,U)\in \set$ to the BSDE $(\xi ,f)$ with $Y\geq \epsilon$ for some $\epsilon >0$. Moreover, it holds $a(t)\leq Y_t\leq b(t)$ and $\int Z\, {\rm d}B$ and $U*\widetilde{\mu}$ are {\rm BMO}$(\PP )$-martingales.
\end{corollary}
\begin{proof}
This can be shown with a similar argument  for the uniqueness  as above: Let $(Y',Z',U')$ be another solution to the BSDE $(\xi ,f)$ with $Y'\geq \epsilon$ for some
$\epsilon >0$. Then $f$ satisfies $({\rm \bf A}_{\rm \bf fin})$ for $a(t)\wedge \epsilon \leq y,\,y+u\leq b(t)\vee |Y'|_{\infty}$; hence the solutions coincide  by comparison.
\end{proof}
\begin{example}
As a special case of Corollary~\ref{boundscor} to be applied in Section~\ref{subsubsec:powut}, setting $K:=\big(\gamma |\varphi |_{\infty}^2\big)/\big(2(1-\gamma )^2\big)$ for some
$\gamma \in (0,1)$ and some predictable and bounded process $\varphi$ we define
\begin{align*}
f_t(y,z,u) &:= \widehat{f}_t(y,z) +\int_E g_t(y,u,e)\, \zeta (t,e)\, \lambda ({\rm d}e)\\
&:=\frac{\gamma}{2(1-\gamma )^2}|\varphi_t|^2y+\int_E \left(\frac{1}{1-\gamma}((u(e)+y)^{1-\gamma}y^{\gamma}-y)-u(e)\right)\, \zeta (t,e)\, \lambda ({\rm d}e).
\end{align*}
From
\(
\frac{\partialup g}{\partialup y}(t,y,u,e) =\left( \frac{u+y}{y}\right)^{1-\gamma}+\frac{\gamma}{1-\gamma}\left( \frac{u+y}{y} \right)^{-\gamma}-\frac{1}{1-\gamma},
\)
we see that $f$ is Lipschitz in $y$ within the truncation bounds. Moreover, $g$ is continuously differentiable with bounded derivatives and we have
\(
\frac{\partialup g}{\partialup u}(t,y,u,e)=\left( \frac{u+y}{y} \right)^{-\gamma}-1>-1,
\) 
for $c \leq y,y+u\leq d$.
\end{example}

\subsection{The case of infinite activity}

For linear generators of the form
\(
f_t(y,z,u):=\alpha_t^0 +\alpha_t y+\beta_t z+\int_E \gamma_t(e)u(e)\, \zeta (t,e)\, \lambda ({\rm d}e),
\)
with predictable coefficients $\alpha^0$, $\alpha$, $\beta$ and $\gamma$, JBSDE solutions can be represented  by an adjoint process. In our context of bounded solutions, one needs rather weak conditions on the adjoint  process. This will be used later on in Section~\ref{sec:appl}. 
The idea of proof is standard, cf.\ \cite[Lem.1.23]{KentiaPhD}  for details.
\begin{lemma} \label{lem:RepSolLinBSDEs}
Let $f$ be a linear generator of the form above and let  $\xi$ be in $ L^{\infty}(\FF_T)$.
\begin{enumerate}
\item Assume that $(Y,Z,U)\in \mathcal{S}^{\infty}\times \mathcal{L}^2(B)\times \mathcal{L}^2(\widetilde{\mu} )$ solves the BSDE $(\xi ,f)$. Suppose that the adjoint process $(\Gamma_s^t)_{s\in[t,T]}:=(\exp (\int_t^s \alpha_u\, {\rm d}u)\mathcal{E}(\int \beta {\rm d}B+\gamma*\widetilde{\mu})_t^s)_{s\in[t, T]}$ is in $\mathcal{S}^1$ for any $t\le T$ and $\alpha^0$ is bounded. Then $Y$ is represented as
\(
Y_t=\mathbb{E}\big[\Gamma_T^t \xi+\int_t^T \Gamma_s^t \alpha_s^0\, {\rm d}s|\FF_t\big].
\)
\item Let $\alpha^0$, $\alpha$, $\beta$ and $\widetilde{\gamma}_t:=\int_E |\gamma_t(e)|^2\, \zeta (t,e)\lambda ({\rm d}e)$, $t\in [0,T]$, be bounded and $\gamma \geq -1$. Then there is a unique solution in $\set$ to the BSDE $(\xi ,f)$ and 
Part~1.\ applies.
\end{enumerate}
\end{lemma}

Our aim is to prove existence and uniqueness beyond Proposition~\ref{finitetheo} for infinite activity of jumps, that means $\lambda (A)$ may be infinite in
(\ref{generator1}).
To show Theorems $\ref{infinitetheogeneral}$ and $\ref{infinitetheo}$, we use a monotone stability approach of \cite{Kobylanski00}: By approximating a generator $f$ of
the form $(\ref{generator})$ (with $A$ such that $\lambda (A)=\infty$) by a sequence $(f^n)_{n\in \NN}$ of the form $(\ref{generator})$ (with $A_n$ such that $\lambda (A_n)<\infty$)
for which solutions' existence is guaranteed, one gets that the limit of these solutions exist and it solves the BSDE with the original data. 
As in \cite{Kobylanski00}, the monotone approximation approach is perceived as being not easy in execution, a main problem usually being to prove strong convergence of the stochastic integral parts for the BSDE.
By Proposition~\ref{strongconv} convergence works for small terminal
condition $\xi$. That is why we can not apply this Proposition directly to data $(\xi ,f^n)_{n\in \NN}$. Instead we sum (converging) solutions for small
${1}/{N}$-fractions of the desired terminal condition. 
This is inspired by the iterative ansatz from \cite{Morlais10} for a particular generator. For our generator family, 
we adapt and elaborate proofs, using e.g.\ a $\mathcal{S}^1$-closeness argument for the proof of the  strong approximation step.
Compared to \cite{Morlais10}, the analysis for our general family of JBSDEs adds clarity and structural 
insight into what is really needed. It extends the scope of the BSDE stability approach \cite{Kobylanski00,Morlais10}, in particular with regards to
non-Lipschitz dependencies in the jump-integrand, while the 
proof shows comparable ease for the (usually laborious) strong approximation step in the setup under consideration. 
Differently to e.g.\ \cite{ElKarouiMatoussiNgoupeyou16,Morlais10,Yao17}, no  exponential transforms or convolutions  are needed here,  as our generators are ``quadratic'' in $U$ but not in $Z$. 
Despite similarities at first sight, a closer look reveals that Theorem~\ref{infinitetheogeneral} is different from  
\cite[Thm.5.4]{KPZ15}, both in the method of proof and in scope: They prove 
existence for small terminal conditions by following the fixed point approach by \cite{Tevzadze08},
whereas we  show stability for small terminal conditions (Proposition~\ref{strongconv}) 
and apply a different  pasting procedure, approximating not only terminal data but also generators. Here wellposedness of the approximating JBSDEs is obtained directly 
from classical theory by using comparison and estimates from Section~\ref{sec:comp}, which 
enable us to argue
within uniform a-priori bounds for the approximating sequence. 
 Examples in Section~\ref{sec:appl} demonstrate that
also the scope of our results is different.

In more detail, the task for the next Theorem~\ref{infinitetheogeneral} is to construct generators $(f^{k,n})_{1\leq k\leq N, n\in \NN}$
and solutions $(Y^{k,n},Z^{k,n},U^{k,n})$ to the BSDEs with data $(\xi /N,f^{k,n})$ for $N$ large enough such that $(Y^{k,n},Z^{k,n},U^{k,n})$ converges if $n \to \infty$ and
$(Y^n,Z^n,U^n):=\sum_{k=1}^N (Y^{k,n},Z^{k,n},U^{k,n})$ solves the BSDE $(\xi ,f^n)$. In this case $(Y^n,Z^n,U^n)$ converges and its limit is a solution candidate for the BSDE $(\xi ,f)$.
For this program,  we next show  a stability result for JBSDE.

\begin{proposition} \label{strongconv}
Let $(\xi^n) \subset L^{\infty}(\FF_T)$ with $\xi^n\rightarrow \xi$ in $L^2(\FF_T)$ and $(f^n)_{n\in \NN}$ be a sequence of generators with $f^n_.(0,0,0)=0,\ \forall n$, having property
$(B_{\gamma^n})$ such that  $K_f^{y,z}:=\sup_{n\in \NN} K_{f^n}^{y,z}<\infty$. Denote by $(Y^n,Z^n,U^n)\in \set$ the solution to the BSDE $(\xi ,f^n)$ with $Y^n$ bounded by
$|\xi |_{\infty}\exp (K_{f^n}^{y,z}T)$ and set $\tilde{c}:=|\xi |_{\infty}\exp (K_f^{y,z}T)$. Assume that $Y^n$ converges pointwise, $(Z^n,U^n)\rightarrow (Z,U)$ converges weakly
in $\mathcal{L}^2(B) \times \mathcal{L}^2(\widetilde{\mu})$ and $|f_t^n(0,0,u)|\leq \widehat{K}|u|_t^2+\widehat{L}_t$  for all $n$ and $u$ with $|u|\leq 2\tilde{c}$, $\widehat{K}\in \RR_+$ and $\widehat{L}\in L^1(\PP\otimes {\rm d}t)$.
Then $(Z^n,U^n)$ converges to $(Z,U)$ strongly in $\mathcal{L}^2(B) \times \mathcal{L}^2(\widetilde{\mu})$,
 if $|\xi |_{\infty}\equiv  \tilde{c}\exp (-K_f^{y,z}T)  \leq {\exp (-K_f^{y,z}T)}/({80 \max \{ K_f^{y,z},\widehat{K}\} })$.
\end{proposition}
\begin{proof}
We note that $(Y^n,Z^n,U^n)$ is uniquely defined by Proposition $\ref{finitetheogeneral}$. To prove strong convergence of $(Z^{n})_{n\in\NN}$ and $(U^{n})_{n\in \NN}$ we consider
$\updelta Y:=Y^{n}-Y^{m}$, $\updelta Z:=Z^{n}-Z^{m}$, $\updelta U:=U^{n}-U^{m}$ and apply It\^o's formula for general semimartingales to $(\updelta Y)^2$ to obtain
\begin{align*}
(\updelta Y_0)^2 =&\; (\updelta Y_T)^2 +\int_0^T 2\updelta Y_{s-} (f_s^{n}(Y_{s-}^{n},Z_s^{n},U_s^{n})-f_s^{m}(Y_{s-}^{m},Z_s^{m},U_s^{m})) {\rm d}s \\
                   & -\int_0^T \lVert \updelta Z_s\rVert^2 {\rm d}s-2\int_0^T \updelta Y_{s-} \updelta Z_s\, {\rm d}B_s 
 -\int_0^T \mskip-10mu \int_E (\updelta Y_{s-} +\updelta U_s(e))^2-(\updelta Y_{s-})^2\, \widetilde{\mu} ({\rm d}s,{\rm d}e)\\
                   & -\int_0^T \mskip-10mu \int_E (\updelta Y_{s-} +\updelta U_s(e))^2-(\updelta Y_{s-})^2-2\updelta Y_{s-}\updelta U_s(e)\, \nu ({\rm d}s,{\rm d}e).
\end{align*}
Noting that the stochastic integrals are martingales one concludes that
 \begin{equation}\label{ineq1}\begin{split}
&\mathbb{E}\Big( \int_0^T 2\updelta Y_{s-}(f_s^{n}(Y_{s-}^{n},Z_s^{n},U_s^{n})-f_s^{m}(Y_{s-}^{m},Z_s^{m},U_s^{m}))\, {\rm d}s \Big)\\
&= \mathbb{E}\Big( \int_0^T \mskip-10mu \int_E \updelta U_s(e)^2\, \nu ({\rm d}s,{\rm d}e) \Big)  +\mathbb{E}\Big( \int_0^T \lVert \updelta Z_s\rVert^2\, {\rm d}s \Big) - \mathbb{E}\big((\updelta Y_T)^2)+\mathbb{E}((\updelta Y_0)^2\big) .
\end{split}
\end{equation}
Using the inequalities $a\leq a^2+\sfrac{1}{4}$, $(a+b)^2\leq 2(a^2+b^2)$, $(a+b+c)^2\leq 3(a^2+b^2+c^2)$, the Lipschitz property of $f^{n}$ in $y$ and $z$ and the estimate for $f_t^n(0,0,u)$, we have
\begin{equation}\label{ineq2}
\begin{split}
&|f_s^{n}(Y_{s-}^{n},Z_s^{n},U_s^{n})-f_s^{m}(Y_{s-}^{m},Z_s^{m},U_s^{m})|\\
&\quad \leq K_{f^{n}}^{y,z} (|Y_{s-}^{n}|+\lVert Z_s^{n}\rVert )+K_{f^{m}}^{y,z}(|Y_{s-}^{m}|+\lVert Z_s^{m}\rVert )+\widehat{K} |U_s^{n}|_s^2+\widehat{L}_s +\widehat{K} |U_s^{m}|_s^2+\widehat{L}_s\\
&\quad \leq K_1+2\widehat{L}_s+K_2(\lVert \updelta Z_s\rVert^2+\lVert Z_s^{n}-Z_s\rVert^2+\lVert Z_s\rVert^2 +|\updelta U_s|_s^2+|U_s^{n}-U_s|_s^2+ |U_s|_s^2),
\end{split}
\end{equation}
where $K_1:=K_f^{y,z}(2\tilde{c}+\sfrac{1}{2})\in \RR$, $K_2:=5\max \{ K_f^{y,z},\widehat{K}\}$ and $\vert\cdot\vert_t$ is defined in (\ref{utnorm}).
Combing inequalities ($\ref{ineq1}$) and ($\ref{ineq2}$) yields
\begin{align*}
\EE \Big( \int_0^T \lVert \updelta Z_s\rVert^2+|\updelta U_s|_s^2\, {\rm d}s \Big) \le  2\EE \Big( \int_0^T &|\updelta Y_{s-}|(K_1+2\widehat{L}_s+K_2(\lVert \updelta Z_s\rVert^2+\lVert Z_s^{n}-Z_s\rVert^2+\lVert Z_s\rVert^2\Big.\\
                                    &\Big.+|\updelta U_s|_s^2+|U_s^{n}-U_s|_s^2+|U_s|_s^2))\, {\rm d}s \Big) + \EE \big((\xi^n-\xi^m)^2\big).
\end{align*}
Let us recall that the predictable projection of $Y$, denoted by $Y^{\rm p}$, is defined as the unique predictable process $X$ such that $X_{\tau}=\EE (Y_{\tau}|\FF_{\tau-})$ on $\{ \tau<\infty \}$ for all predictable times $\tau$. For $Y^{n}$ it holds $(Y^{n})^{\rm p}=Y_-^{n}$. This follows from \cite[Prop.I.2.35.]{JacodShiryaev03}
using that $Y^n$ is c\`adl\`ag, adapted and quasi-left-continuous,  as $\Delta Y_{\tau}=\Delta U*\widetilde{\mu}_{\tau}=0$ on $\{ \tau<\infty \}$ holds for all predictable times $\tau$ thanks to the absolute continuity of
the compensator $\nu$. Noting that $1-2K_2|\updelta Y_{s-}|\geq 1-4K_2\tilde{c}\geq 3/4$ and setting $Y:=\lim_{n\rightarrow \infty}Y^n$ we deduce by the weak convergence of $(Z^{n})_{n\in \NN}$ and $(U^{n})_{n\in \NN}$, $Y_{-}^{n}=(Y^{n})^{\rm p} \uparrow (Y)^{\rm p}$ as $n\rightarrow \infty$ and by Lebesgue's dominated convergence theorem
\begin{align*}
& \frac{3}{4} \mathbb{E}\Big( \int_0^T  \lVert Z_s^{n}-Z_s\rVert^2+|U_s^{n}-U_s|_s^2\, {\rm d}s\Big)\\
& \leq \frac{3}{4}\liminf\limits_{m\rightarrow \infty} \mathbb{E}\Big( \int_0^T \lVert Z_s^{n}-Z_s^{m}\rVert^2+|U_s^{n}-U_s^{m}|_s^2\, {\rm d}s \Big) \\
& \leq \liminf\limits_{m\rightarrow \infty} 2\EE \Big( \int_0^T |\updelta Y_{s-}|(K_1+2\widehat{L}_s+K_2(\lVert Z_s^{n}-Z_s\rVert^2+\lVert Z_s\rVert^2+|U_s^{n}-U_s|_s^2+|U_s|_s^2))\, {\rm d}s \Big)\\
&\qquad\qquad+\EE \big((\xi^m-\xi^n)^2\big) \\
& = 2\mathbb{E}\Big( \int_0^T |Y_{s-}^{n}-(Y_s)^{\rm p}|( K_1+2\widehat{L}_s+K_2(\lVert Z_s^{n}-Z_s\rVert^2+\lVert Z_s\rVert^2+|U_s^{n}-U_s|_s^2+|U_s|_s^2))\, {\rm d}s\Big)\\
&\quad+\EE \big((\xi -\xi^n)^2\big).
\end{align*}
Noting $\sfrac{3}{4}-2K_2|Y_{s-}^{n}-(Y_s)^{\rm p}|\geq \sfrac{3}{4}-4K_2\tilde{c}\geq \sfrac{1}{2}$, one obtains with dominated convergence 
\begin{align*}
&\frac{1}{2}\limsup\limits_{n\rightarrow \infty} \mathbb{E}\Big( \int_0^T \lVert Z_s^{n}-Z_s\rVert^2+|U_s^{n}-U_s|_s^2\,  {\rm d}s\Big) \\
&\qquad \leq \limsup\limits_{n\rightarrow \infty} 2\EE \Big( \int_0^T |Y_{s-}^{n}-(Y_s)^{\rm p}|(K_1+2\widehat{L}_s+\lVert Z_s\rVert^2+|U_s|_s^2)\, {\rm d}s \Big) +\EE \big((\xi^n-\xi )^2\big)=0.
\end{align*}
\end{proof}
We will need the following result which is a slight variation of \cite[Lem.2.5]{Kobylanski00}.
\begin{lemma} \label{helplemma2}
Let $(Z^n)_{n\in \NN}$ be convergent in $\mathcal{L}^2(B)$ and $(U^n)_{n\in \NN}$ convergent in $\mathcal{L}^2(\widetilde{\mu})$. Then
there exists a subsequence $(n_k)_{k\in \NN}$ such that
\(
\sup\limits_{n_k} \lVert Z^{n_k} \rVert \in L^2(\PP \otimes {\rm d}t) \; \mbox{and}
\; \sup\limits_{n_k}  |U_t^{n_k}|_t  \in L^2(\PP \otimes {\rm d}t).
\)
\end{lemma}
\begin{proof}
The result for $(Z^n)_{n\in \NN}$ is from~\cite{Kobylanski00} and  the argument for $(U^n)_{n\in \NN}$ is analogous.
\end{proof}

\begin{theorem}[Monotone stability, infinite activity] \label{infinitetheogeneral}
Let $\xi \in L^{\infty}(\FF_T)$ and let $(f^n)_n$ be a sequence of generators satisfying condition~$(B_{\gamma^n})$ with $K_f^{y,z}:=\sup_{n\in\NN}K_{f^n}^{y,z}<\infty$. Assume that
\begin{enumerate}\item[1.]\label{zerosolution} there is $(\widehat{Y},\widehat{Z},\widehat{U})$ in $\set$ with $\widehat{U}$ bounded and  $f_t^n(\widehat{Y}_{t-},\widehat{Z}_t,\widehat{U}_t)\equiv 0$ for all $n$,
\item[2.]\label{quadbound} for all $u\in L^0(\mathcal{B}(E),\lambda)$ with $|u|\leq |\widehat{U}|_{\infty}+2|\xi |_{\infty}\exp (K_f^{y,z}T)$ there exists $\widehat{K}\in \RR_+$ and a process $\widehat{L}\in L^1(\PP\otimes {\rm d}t)$ such that 
$|f_t^n(0,0,u)|\leq \widehat{K}|u|_t^2+\widehat{L}_t$ for each $n\in \NN$,
\item[3.] the sequence $(f^n)_{n\in\NN}$ converges pointwise and monotonically to a generator $f$,
\item[4.]\label{BMOprop} there is a ${\rm BMO}(\PP )$-martingale $M$ such that for all truncated generators $f_t^{n,\hat{c}}(y,z,u):=f_t^n\big((y\vee (-\hat{c}))\wedge \hat{c},z,(u\vee (-2\hat{c}))\wedge (2\hat{c})\big)$ with $\hat{c}:= |\widehat{Y}|_{\infty}+(|\widehat{U}|_{\infty}/2)+\exp (K_f^{y,z}T)|\xi |_{\infty}$ holds
$\int_t^T f_s^{n,\hat{c}}(Y_{s-},Z_s,U_s)\, {\rm d}s\leq \langle M\rangle_T-\langle M\rangle_t$ or $-\int_t^T f_s^{n,\hat{c}}(Y_{s-},Z_s,U_s)\, {\rm d}s\leq \langle M\rangle_T-\langle M\rangle_t$ for all $n\in \NN$, $(Y,Z,U)\in \set$, and
\item[5.] \label{convprop} for all $(Y,Z,U)\in \set$ and $(U^n)_{n\in \NN}\in \mathcal{L}^2(\widetilde{\mu})$ with $U^n\rightarrow U$ in $L^2(\widetilde{\mu})$ it holds $f^n(Y_{-},Z,U^n) \longrightarrow f(Y_{-},Z,U)$ in $L^1(\PP\otimes {\rm d}t)$.
\end{enumerate}
Then
\begin{enumerate}
\renewcommand{\labelenumi}{\roman{enumi}}
\item $\mskip-10mu )$ there exists a solution $(Y,Z,U)\in \set$ for the BSDE $(\xi ,f)$, with $\int Z\, {\rm d}B$ and $U*\widetilde{\mu}$ being {\rm BMO}$(\PP )$-martingales, and
\item $\mskip-10mu )$ this solution is unique if additionally $f$ satisfies condition $({\rm \bf A}'_{\bm \gamma})$.
\end{enumerate}
\end{theorem}
\begin{proof}
Let us first outline the overall program of the proof. We want to construct generators $(f^{k,n})_{1\leq k\leq N, n\in \NN}$ and solutions $(Y^{k,n},Z^{k,n},U^{k,n})$ to the BSDEs
$(\xi /N,f^{k,n})$ for $N$ sufficiently large (to employ Proposition~\ref{strongconv} and get that $((Y^{k,n},Z^{k,n},U^{k,n}))_{n\in \NN}$ converges and $(Y^n,Z^n,U^n):=\sum_{k=1}^N (Y^{k,n},Z^{k,n},U^{k,n})$
solves the BSDE $(\xi ,f^n)$). We show that if for some $k<N$ and all $1\leq l\leq k$ and $n\in \NN$ we have already constructed generators $(f^{l,n})_{1\leq l\leq k,n\in \NN}$
such that there exists solutions $((Y^{l,n},Z^{l,n},U^{l,n}))_{n\in \NN}$ to the BSDEs $(\xi /N,f^{l,n})$ converging for $n\rightarrow \infty$, with $|Y^{l,n}|_{\infty}\leq \exp (K_f^{y,z}T)|\xi |_{\infty}/N=:\tilde{c}$,
then for $\overline{Y}^{k,n}:=\widehat{Y}+\sum_{l=1}^k Y^{l,n}$ with $\overline{Z}^{k,n}$ and $\overline{U}^{k,n}$ defined analogously and
\begin{align} \label{induct}
f_t^{k+1,n}(y,z,u):= f_t^n\big(y+\overline{Y}_{t-}^{k,n},z+\overline{Z}_t^{k,n},u+\overline{U}_t^{k,n}\big)-f_t^n\big(\overline{Y}_{t-}^{k,n},\overline{Z}_t^{k,n},\overline{U}_t^{k,n}\big)
\end{align}
there are solutions $(Y^{k+1,n},Z^{k+1,n},U^{k+1,n})\in\set$ to the BSDEs $(\xi /N,f^{k+1,n})$, converging (in $n$) and satisfying $|Y^{k+1,n}|_{\infty}\leq \tilde{c}$.
Starting initially with the triple $(Y^{0,n},Z^{0,n},U^{0,n})$ defined by $(Y^{0,n},Z^{0,n},U^{0,n}):=(\widehat{Y},\widehat{Z},\widehat{U})$, formula~$(\ref{induct})$ gives an inductive
construction of the generators $f^{k,n}$ and triples $(Y^{k,n},Z^{k,n},U^{k,n})\in\set$ solving the BSDE $(\xi /N,f^{k,n})$ and converging for $n\rightarrow \infty$ with
$|Y^{k,n}|_{\infty}\leq \tilde{c}$ for each $n\in N$ and $1\leq k\leq N$.\\
Note that $f^{k+1,n}$ is Lipschitz continuous in $y$ and $z$ with Lipschitz constant $K_{f^n}^{y,z}$, locally Lipschitz in $u$ and satisfies condition $(A_{\gamma^{k+1,n}})$ with
\[
\gamma_s^{k+1,n}(y,z,u,u',e) := \gamma_s^n\big(y+\overline{Y}^{k,n}_{s-},z+\overline{Z}^{k,n}_s,u+\overline{U}_s^{k,n}(e),u'+\overline{U}_s^{k,n}(e),e\big)
\]
and $f_t^{k+1,n}(0,0,0)\equiv 0$. Hence by the existence and uniqueness result for the finite activity case (see Proposition~\ref{finitetheogeneral}), there exists a unique solution
$(Y^{k+1,n},Z^{k+1,n},U^{k+1,n})$ to the BSDE $(\xi /N,f^{k+1,n})$ such that $Y^{k+1,n}$ is bounded by $\tilde{c}$.\\
To apply Proposition~\ref{strongconv}, we have to check that the sequence $(Y^{k+1,n})_{n\in \NN}$ converges pointwise, that $(Z^{k+1,n},U^{k+1,n})_{n\in \NN}$ converges weakly in
$\mathcal{L}^2(B)\times \mathcal{L}^2(\widetilde{\mu})$ and that $f^{k+1,n}(0,0,u)$ can be locally bounded by an affine function in  $|u|^2$. 
Having telescoping sums in $(\ref{induct})$ implies that $(\overline{Y}^{l,n},\overline{Z}^{l,n},\overline{U}^{l,n})$ solves the BSDE $(\widehat{Y}_T+l\xi/N ,f^n)$. 
By the comparison
result of Proposition~\ref{comparegeneral}, the sequences $(\overline{Y}^{k,n})_{n\in \NN}$ and $(\overline{Y}^{k+1,n})_{n\in \NN}$ are monotonic (and bounded) in $n$ so
that finite limits $\lim_{n\rightarrow \infty}Y^{k+1,n}=\lim_{n\rightarrow \infty}\overline{Y}^{k+1,n}-\lim_{n\rightarrow \infty}\overline{Y}^{k,n}$ exists, $\PP\otimes {\rm d}t$-a.e..
By Lemma~\ref{BMOproperty}, the sequences $(\overline{Z}^{k,n},\overline{U}^{k,n})_{n\in\NN}$ and $(\overline{Z}^{k+1,n},\overline{U}^{k+1,n})_{n\in\NN}$ are bounded in $\mathcal{L}^2(B)\times \mathcal{L}^2(\widetilde{\mu})$;
hence $(Z^{k+1,n},U^{k+1,n})$ is weakly convergent in $\mathcal{L}^2(B)\times \mathcal{L}^2(\widetilde{\mu})$ along a subsequence which we still index by $n$  for simplicity.
Due to the Lipschitz continuity of $f^n$ and Assumption $2.$, we get for all $|u|\leq 2\tilde{c}$ that
\begin{align*}
\big|f_t^{k+1,n}(0,0,u)\big| &\leq \big|f_t^n\big(\overline{Y}_{t-}^{k,n},\overline{Z}_t^{k,n},u+\overline{U}_t^{k,n}\big)-f_t^n\big(\overline{Y}_{t-}^{k,n},\overline{Z}_t^{k,n},\overline{U}_t^{k,n}\big)\big| & \\
&\leq 2K_{f^n}^{y,z}\big(|\overline{Y}_{t-}^{k,n}|+\lVert \overline{Z}_t^{k,n}\rVert\big) +\widehat{K}\big(|u+\overline{U}_t^{k,n}|_t^2+|\overline{U}_t^{k,n}|_t^2\big)+2\widehat{L}_t \\
&\leq 2\widehat{K}|u|_t^2 +\widetilde{L}_t,
\end{align*}
where $\widetilde{L}_t=2K_f^{y,z}(\hat{c}+\sup_{n\in \NN}\lVert \overline{Z}_t^{k,n}\rVert^2+\sfrac{1}{4})+3\widehat{K}\sup_{n\in\NN}|\overline{U}_t^{k,n}|_t^2+2\widehat{L}_t$. Here we used that by induction hypothesis $(\overline{Z}^{k,n},\overline{U}^{k,n})_n$ is convergent so 
that $\sup_{n\in \NN}(\lVert \overline{Z}_t^{k,n}\rVert^2+|\overline{U}_t^{k,n}|_t^2)$ is $\PP\otimes {\rm d}t$-integrable by Lemma~\ref{helplemma2} along a subsequence which again for simplicity we still index by $n$.
This implies that $\widetilde{L}\in L^1(\PP\otimes {\rm d}t)$, and therefore by Proposition~\ref{strongconv}, the sequence $(Z^n,U^n):=(\overline{Z}^{N,n},\overline{U}^{N,n})$
converges in $\mathcal{L}^2(B)\times \mathcal{L}^2(\widetilde{\mu})$ to some $(Z,U)$ in $\mathcal{L}^2(B)\times \mathcal{L}^2(\widetilde{\mu})$ while $(Y^n):=(\overline{Y}^{N,n})$ converges to some
$Y$. Hence, $f^n(Y^n_-,Z^n,U^n)-f^n(Y_-,Z,U^n)$ converges to $0$ in $L^1(\PP\otimes {\rm d}t)$ and by Assumption $5.$ we have $f^n(Y^n_-,Z^n,U^n)\rightarrow f(Y_-,Z,U)$ in $L^1(\PP\otimes {\rm d}t)$.
The stochastic integrals $(Z^n-Z^m)\mal B$ and $(U^n-U^m)*\widetilde{\mu}$ belong to $\mathcal{S}^2\subset \mathcal{S}^1$ by Doob's inequality, with $\mathcal{S}^1$-norms being bounded by a multiple of $\lVert Z^n-Z^m\rVert_{\mathcal{L}^2(B)}$ and $\lVert U^n-U^m\rVert_{\mathcal{L}^2(\widetilde{\mu})}$ respectively.
Since $|Y^n-Y^m|_{\mathcal{S}^1}$ is dominated by
\[
\lVert f^n(Y^n_-,Z^n,U^n)-f^m(Y^m_-,Z^m,U^m)\rVert_{L^1(\PP\otimes {\rm d}t)} +C(\lVert Z^n-Z^m\rVert_{\mathcal{L}^2(B)}+\lVert U^n-U^m\rVert_{\mathcal{L}^2(\widetilde{\mu})})
\]
for some constant $C>0$ with the bound tending to $0$ as $n,m\rightarrow 0$, we can take $Y$ in $\mathcal{S}^1$ due to completeness of $\mathcal{S}^1$; see \cite[VII. 3,64]{DellacherieMeyer82}.\\
Finally, $(Y,Z,U)$ solves the BSDE $(\xi ,f)$ since the approximating solutions $(Y^n,Z^n,U^n)_{n\in\NN}$ of the BSDE $(\xi ,f^n)_{n\in\NN}$ converge to some $(Y,Z,U)\in\set$ and $f^n(Y^n_-,Z^n,U^n)$
tends to $f(Y_-,Z,U)$ in $L^1(\PP\otimes {\rm d}t)$. Hence, we have $\int_0^t f_s^n(Y^n_{s-},Z_s^n,U_s^n){\rm d}s \rightarrow \int_0^t f_s(Y_{s-},Z_s,U_s){\rm d}s$, $\int_0^t Z_s^n {\rm d}B_s\rightarrow \int_0^t Z_s {\rm d}B_s$ and
$U^n*\widetilde{\mu}_t \rightarrow U*\widetilde{\mu}_t$ $\PP$-a.s.\  (along a subsequence) for all $0\leq t\leq T$.
\end{proof}
The next corollary to Theorem \ref{infinitetheogeneral} provides conditions under which the $Z$-component of the JBSDE solution vanishes. Such is useful for applications   in a pure-jump context (see e.g.\ Section~\ref{subsubsec:CaseDiscAssetPrice} or \cite{CFJ16}) with weak PRP by $\widetilde{\mu}$ alone (cf.\ Example~\ref{example_WPRP}, Parts 1.,3.,4.), without a Brownian motion.
Clearly an independent Brownian motion can always be added by enlarging the probability space, but it is then natural to ask for a JBSDE solution with trivial $Z$-component, adapted to the original filtration. 
Instead of re-doing the entire argument leading to Theorem~\ref{infinitetheogeneral} but now for JBSDEs solely driven by a random measure $\widetilde{\mu}$ with generators without a $z$-argument, the next result gives a direct argument to this end.
An example where the corollary is applied is given in Section~\ref{subsubsec:CaseDiscAssetPrice}.
\begin{corollary}\label{cor:JBSDEZzero}
 Let $\mu=\mu^X$ be the random measure associated to a pure-jump process $X$, such that the compensated random measure $\widetilde{\mu}$ alone has the weak PRP (see (\ref{WPRP}))
 with respect to the usual filtration $\filt^X$ of $X$. Let $B$ be a $d$-dimensional Brownian motion independent of $X$.
With respect to $\filt:=\filt^{B,X}$, let $f, (f^n)_n,\xi$ satisfy the assumptions
 of Theorem \ref{infinitetheogeneral} with $\widehat{Z}=0$ and $f$ satisfying $({\rm \bf A}'_{\bm \gamma})$.
Let $\xi$ be in $ L^\infty(\FF^X_T)$ and $f,f^n$ be $\mathcal{P}(\filt^X)\otimes \mathcal{B}(\RR^{d+1})\otimes\mathcal{B}(L^0(\mathcal{B}(E)))$-measurable.
 Then the JBSDE $(\xi,f)$ admits a unique solution $(Y,Z,U)$ in $\set$, and we have that $Y$ is $\filt^X$-adapted,
$Z=0$,  and 
$U$ can be taken as measurable with respect to $\widetilde{\mathcal{P}}(\filt^X)$.
\end{corollary}
\begin{proof}
 Let $B'$ be a ($1$-dimensional) Brownian motion independent of $(B,X)$. Then $\bar{B} := (B,B')$ is a $(d+1)$-dimensional Brownian motion independent of $X$. Let
 $\filt':=\filt^{B',X}$ and $\bar{\filt}:=\filt^{\bar{B},X}$ denote the usual filtrations of $(B',X)$ and $(\bar{B},X)$. As in Example \ref{example_WPRP}.3., $(B,\widetilde{\mu})$,
 $(B',\widetilde{\mu})$ and $(\bar{B},\widetilde{\mu})$ each admits the weak PRP w.r.t.\  $\filt,\filt'$ and $\bar{\filt}$ respectively.
 Now consider the generator function $\widetilde{f}$ that does not depend on $z$ and is defined by $\widetilde{f}_t(y,u) := f_t(y,0,u)$. Because $\widehat{Z}=0$, the conditions for Theorem \ref{infinitetheogeneral}
 are met by $\widetilde{f}^n:=f^n(\cdot,0,\cdot)$. In addition, $\widetilde{f}$ satisfies condition $({\rm \bf A}'_{\bm \gamma})$ since $f$ does. Since $\xi$ is $\FF^X_T$-measurable and $\widetilde{f}$ is $\mathcal{P}(\filt^X)\otimes \mathcal{B}(\RR)\otimes\mathcal{B}(L^0(\mathcal{B}(E)))$-measurable, 
 then by Theorem~\ref{infinitetheogeneral} the JBSDE $(\xi,\widetilde{f})$ 
simultaneously admits unique solutions $(Y,Z,U)$, $(Y',Z',U')$ and $(\bar{Y},\bar{Z},\bar{U})$
 in the respective $\mathcal{S}^\infty\times \mathcal{L}^2(\cdot)\times\mathcal{L}^2(\widetilde{\mu})$-spaces for each of the filtrations $\filt,\filt'$ and $\bar{\filt}$. 
 Noting that both $\filt$ and $\filt'$ are sub-filtrations of $\bar{\filt}$, we get by uniqueness of $(\bar{Y},\bar{Z},\bar{U})$ that $Z\mal B = Z'\mal B' = \bar{Z}\mal\bar{B}$ and that $Y$ is $\filt^X$-adapted. The former implies
$Z=Z'=0$ by the strong orthogonality of $B$ and $B'$. The claim follows, by noting that the JBSDE gives the (unique) canonical decomposition of the special semimartingale $Y$  and using weak predictable martingale representation in $\filt^X$.
\end{proof}

A natural ansatz to approximate an $f$ of the form $(\ref{generator})$ with $\lambda (A)=\infty$ is by taking
\begin{align} \label{fapprox}
f_t^n(y,z,u):= \widehat{f}_t(y,z)+\int_{A^n} g_t(u(e),e)\, \zeta (t,e)\, \lambda ({\rm d}e),
\end{align}
for an increasing sequence $(A_n)_{n\in \NN}\uparrow A$ of measurable sets with $\lambda (A_n)<\infty$ (as $\lambda$ is $\sigma$-finite).
\begin{theorem}[Wellposedness, infinite activity of jumps]   \label{infinitetheo}
Let the generator $f$ of the JBSDE be of the form (\ref{generator}) and let $\xi$ be in $ L^{\infty}(\FF_T)$. Let $\widehat{f}$ be Lipschitz continuous with respect to $(y,z)$ uniformly in $(\omega,t,u)$, and 
let $u\mapsto g(t,u,e)$ be absolutely continuous in $u$, for all $(\omega ,t,e)$, with its density function 
$g'(t,u,e)$ being  
 strictly greater than $-1$ and locally bounded (in $u$) from above. 

Assume that 
\begin{enumerate}\item[1.]\label{zerosolution2} there exists $(\widehat{Y},\widehat{Z},\widehat{U})\in \set$ with $|\widehat{U}|_{\infty}<\infty$, $\widehat{f}_t(\widehat{Y}_t,\widehat{Z}_t)\equiv 0$, $g_t(\widehat{U}_t(e),e)\equiv 0$,
\item[2.] the function $g$ is locally bounded in $|u|^2$ uniformly in $(\omega ,t,e)$, i.e.\ locally in $u$ (for any bounded neighborhood $N$ of $0$) there exists a $K>0$ such that $|g_t(u,e)|\leq K|u|^2$ (for all $u\in N$),
\item[3.] and there exists $D:\RR \mapsto \RR$ continuous such that 
either  
$g\geq 0$ and 
$\widehat{f}_t(y,z)\geq D(y)$ for $|y|\leq \hat{c}:=|\widehat{Y}|_{\infty}+(|\widehat{U}|_{\infty}/2)+|\xi |_{\infty}\exp (K_{\widehat{f}}^{y,z}T)$, or $g\leq 0$ and $\widehat{f}_t(y,z)\leq D(y)$ for $|y|\leq \hat{c}$.
\end{enumerate}
Then
\begin{enumerate}
\renewcommand{\labelenumi}{\roman{enumi}}
\item $\mskip-10mu )$  there exists a solution $(Y,Z,U)\in \set$ to the JBSDE and for each solution triple the stochastic integrals $\int Z\, {\rm d}B$ and $U*\widetilde{\mu}$ are {\rm BMO}-martingales, and
\item $\mskip-10mu )$ 
this solution is unique if  moreover the function $g$ satisfies condition $({\rm \bf A}_{\rm \bf infi})$. 
\end{enumerate}
Finally, the same statements hold if condition $1.$ is replaced by assuming that $f$ is not depending on $y$, i.e.\ $f_t(y,z,u)=f_t(z,u)$, and that $\widehat{f}$ is bounded.
\end{theorem}
\begin{proof}
We check that the assumptions of Theorem~\ref{infinitetheogeneral} are satisfied.
Clearly conditions $1.$ and $2.$ are sufficient for assumptions $1.$ and $2.$ in Theorem~\ref{infinitetheogeneral}. The $f^n$ given by $(\ref{fapprox})$ satisfy conditions~$(B_{\gamma^n})$ (cf.\ Example~\ref{exagamma} and note $\lambda(A_n)<\infty$) and the sequence $(f^n)$ is
either monotone increasing or monotone decreasing, depending on the sign of $g$. For the next assumption $4.$, $f^{n,\hat{c}}$ is bounded from above (or resp. below) by $\sup_{|y|\leq \hat{c}}D(y)$
(respectively $\inf_{|y|\leq \hat{c}}D(y)$).
To show that also condition $5.$ of Theorem~\ref{infinitetheogeneral} holds, we prove that $g_t(U_t^n(e),e)\,\mathds{1}_{A_n}(e)$ converge to $g_t(U_t(e),e)$ in $L^1(\PP\otimes \nu)$
as $n\to \infty$ for $U^n\rightarrow U$ in $\mathcal{L}^2(\widetilde{\mu})$, recalling $(\ref{nuzetadensity})$. We set $B_n:=\big(g_t(U_t^n(e),e)-g_t(U_t(e),e)\big)\,\mathds{1}_{A_n}(e)$ and $C_n:=g_t(U_t(e),e)\,\mathds{1}_{A_n^c}(e)$.
Both sequences $(B_n)_{n\in\NN}$ and $(C_n)_{n\in\NN}$ converge to $0$ $\PP\otimes \nu$-a.e.\  since $U^n\rightarrow U$ in $L^2(\PP\otimes \nu)$, $g$ is locally Lipschitz in $u$ and
$A_n^c\downarrow \emptyset$. Moreover, they are bounded by integrable random variables. In particular, $B_n$ is bounded by $\widehat{K}\big(\sup_{n\in\NN}|U_t^n(e)|^2+|U_t(e)|^2\big)$ for some
$\widehat{K}>0$ which is integrable along a subsequence due to Lemma~\ref{helplemma2}. Hence applying the dominated convergence theorem yields the desired result.\\
In the alternative case without the Assumption $1.$, existence is still guaranteed. Indeed, let $f_t(y,z,u)=f_t(z,u)$ and $\widehat{f}$ be bounded. Denoting $\widetilde{f}_t(z,u):=f_t(z,u)-f_t(0,0)$ and
$\widetilde{\xi}:=\xi +\int_0^T f_t(0,0)\, {\rm d}t$, there exists a unique solution $(\widetilde{Y},Z,U)$ in $\set$ to the BSDE $(\widetilde{\xi},\widetilde{f})$ with $\int Z\, {\rm d}B$ and $U*\widetilde{\mu}$ being
BMO-martingales by the first version of this theorem and noting that $g_t(0,e)\equiv 0$ and $f_t(0,0)=\widehat{f}_t(0)$ is bounded. Taking $Y_t:=\widetilde{Y}_t-\int_0^t \overline{f}_s(0,0)\, {\rm d}s$,
we obtain that $(Y,Z,U)$ solves the BSDE with the data $(\xi ,f)$. If moreover the function $f$ satisfies $({\rm \bf A}_{\rm \bf infi})$, then $f$ satisfies $({\rm \bf A}'_{\bm \gamma})$ (cf.\ Example \ref{exagamma}.3.) and 
hence uniqueness follows from applicability of the comparison argument in Proposition~\ref{comparegeneral}.
\end{proof}

\begin{example}
A function $g$ is locally bounded in $|u|^2$ in the sense of condition 2. in Theorem~\ref{infinitetheo} if, for instance, $ u\mapsto g_t(u,e)$ is twice differentiable for any $(\omega ,t,e)$, 
with the second derivative in $u$ being locally bounded uniformly in $(\omega ,t,e)$, and $g_t(0,e)\equiv g'_t(0,e)\equiv 0$ vanishing.
\end{example}

\begin{example}
An example for a  generator that satisfies the assumption of Theorem~\ref{infinitetheo} but has super-exponential growth 
 is
$f$ of the form (\ref{generator}) with $\widehat{f}\equiv0$ and $g_t(u)=\exp(|u^+|^2)-1$.
 Here exists, in general,  no $\gamma\in (0,\infty)$  
 such that $-\frac{1}{\gamma}(e^{-u/\gamma}+u/\gamma-1)\le g_t(u) \le  \frac{1}{\gamma}(e^{u/\gamma}-u/\gamma-1)$ holds for all $u$ and $t$. Thus, the example appears not to satisfy exponential growth assumptions as formulated, e.g., in \cite{AntonelliMancini16}[Assumption (H), Thm.1], \cite {ElKarouiMatoussiNgoupeyou16}[2.condition, Def.5.6]
or \cite{KPZ15}[Assumption~3.1].
\end{example}

Note that convexity is not required for our theorems on comparison, existence and uniqueness for JBSDEs.
Many relevant applications are convex in nature but not all, see examples in Section~\ref{subsubsec:CaseDiscAssetPrice}.

\section{Examples and applications: optimal control in finance}
 \label{sec:appl}

 Results for JBSDEs in the  
literature commonly rely on combinations of several, often quite technical,  assumptions. But their scope can be difficult to judge at first sight without examples,  and to verify them may be not easy. 
This section discusses key applications that JBDEs have found in mathematical finance, and it illustrates
 by concrete examples  the applicability and the scope of the theory from previous sections.
The examples do also help to shed some light on connections and differences to related  literature.
Counter examples might caution against potential pitfalls.

The applications in Section~\ref{subsect:ExpoUtilMax} are about exponential utility maximization, possibly with an additive liability or non-convex constraints.
This problem  is closely related to the entropic risk measure and to (exponential) utility indifference valuation; It has indeed been a standard motivation for much of the (quadratic, non-Lipschitz)  
JBSDE theory, cf.\  \cite{Becherer06,BarrieuElKaroui09,Morlais09,Becherer10uiv,LaevenStadje14,KPZ15}. 
A result on existence of a solution for the specific  
JBSDE of this application has been presented in \cite{Morlais10}, being more general in some 
 aspects (jump-diffusion stock price)  but less so in others (multiple assets, time-inhomogeneous $\mu$).  
Section~\ref{subsubsec:powut} shows how a change of coordinates can transform a JBSDE, which arises from an optimal control problem for power utility maximization
but appears to be out of scope at first, into a JBSDE for which theory of Section~\ref{sec:exandu} can be applied to derive optimal controls and fully characterize the solution to the control problem by JBSDE solutions, 
like in \cite{Huetal05,HuLiangTang18}, by using martingale optimality principles. To our best knowledge, the considered power-utility problem with jumps and a multiplicative liability is solved for the first time in this spirit.
Finally, Section~\ref{subsec:gdprice}
derives  JBSDE solutions  for the no-good-deal valuation problem
  in incomplete 
 markets, which is posed over a multiplicatively stable sub-family of arbitrage-free pricing measures. 
 Also here, where the (non-linear) JBSDE generator is even Lipschitz, the slight generalization of
 Proposition~\ref{comparegeneral} to the classical 
comparison result by \cite{Royer06} is useful;
Indeed, the process $\gamma$ in (\ref{eq:IneqonGenwithGammadiffFromRoyer}) is such
that the martingale condition (\ref{gamma**}) for Proposition~\ref{comparegeneral} can be readily verified,
 while the same appears not  clear for  condition $({\rm \bf A}_{\bm \gamma})$  in \cite[Thm.2.3]{Royer06}.

Sections~\ref{subsubsec:CaseContAssetPrice}, \ref{subsubsec:powut} and \ref{subsec:gdprice}  consider models for a financial market within the framework of Section~\ref{sec:prelim}, consisting of one savings account with zero interest rate (for simplicity)  and $k$ risky assets ($k\leq d$), whose discounted prices  
evolve according to the stochastic differential equation
\begin{align}\label{eq:continuousStock}
{\rm d}S_t={\rm diag}(S_t^i)_{1\leq i\leq k}\sigma_t (\varphi_t {\rm d}t+{\rm d}B_t)=:{\rm diag}(S_t){\rm d}R_t, \quad t\in [0,T], 
\end{align}
with $S_0\in (0,\infty)^k$, where the market price of risk $\varphi$ is a predictable $\mathbb{R}^d$-valued process,
with  $\varphi_t \in \mathrm{Im}\, \sigma_t^{T}=(\mathrm{Ker}\, \sigma_t)^{\bot}$ for all $t\leq T$,
 and $\sigma$ is a predictable $\mathbb{R}^{k\times d}$-valued process such that $\sigma$ is of full rank $k$ (i.e.\ $\det (\sigma_t\sigma_t^{T})\neq 0$ $\PP\otimes {\rm d}t$-a.e.) and integrable w.r.t.\ 
\(
\widehat{B}:=B+\int_0^\cdot \varphi_t\, {\rm d}t.
\)
We take the market price of risk $\varphi$ to be bounded $\mathbb{P}\otimes {\rm d}t$-a.e..
The market is free of arbitrage in the sense that the set ${\cal M}^{\rm e}$ of equivalent local martingale measures for $S$
is non-empty. In particular,  ${\cal M}^{\rm e}$ contains the minimal martingale measure
\begin{align}
{\rm d}\widehat{\mathbb{P}}:=\mathcal{E}\left( -\varphi\mal B \right)_T {\rm d}\mathbb{P}=\exp \Big( -\varphi\mal B_T-\frac{1}{2}\int_0^T |\varphi_t|^2\, {\rm d}t \Big) {\rm d}\mathbb{P}, \label{newmeasure}
\end{align}
under which $\widehat{B}$ is a Brownian motion and $S$ is a local martingale by Girsanov's theorem.
Clearly, the market (\ref{eq:continuousStock}) is  incomplete in general (even if $k=d$ and $\sigma$ is invertible, when the random measure is not trivial, filtration then being non-Brownian), cf.\  Example~$\ref{example_WPRP}$.

\subsection{Exponential utility maximization}\label{subsect:ExpoUtilMax}

For a market with stock prices as in (\ref{eq:continuousStock}), consider the expected utility maximization problem
\begin{align} \label{utilitymax1}
v_t(x) =\esssup_{\theta\in \Theta} \EE \big(u\big(X_T^{\theta,t,x}-\xi \big)|\FF_t\big), \quad t\le T,\ x\in \RR,
\end{align}
for the exponential utility function $u(x):=-\exp (-\alpha x)$ with absolute risk aversion parameter $\alpha >0$, 
with some additive liability $\xi$ and for wealth processes $X^{\theta,t,x}$ of admissible trading strategies $\theta$ as defined below.
 We are going to show, how the value process $v$ 
and optimal trading strategy $\theta^*$ for the problem (\ref{utilitymax1}) can be fully described by JBSDE solutions for two distinct problem cases.

\subsubsection{Case with continuous price processes of risky assets}\label{subsubsec:CaseContAssetPrice}

The set of available trading strategies $\Theta$ consists of all $\mathbb{R}^d$-valued, predictable, S-integrable processes $\theta$ for which the following two conditions are satisfied:
$\mathbb{E}( \int_0^T |\theta_t|^2\, {\rm d}t )$ is finite,
and the family  $\big\{ \exp( -\alpha \int_0^{\tau} \theta_t {\rm d}\widehat{B}_t )\,\big| \,\tau \mbox{ stopping time, }\tau \leq T \big\}$ of random variables is uniformly integrable under $\PP$.
Starting from initial capital $x\in \RR$ at some time $t\le T$, the wealth process  corresponding to investment strategy $\theta\in \Theta$ is given by $X_s^{\theta}=X_s^{\theta,t,x}=x+\int_t^s \theta_u\, {\rm d}\widehat{B}_u$, $s\in [t,T]$.

For this subsection, we assume $k=d$ (so $f$ will not be quadratic in $z$). Let $(Y,Z,U)$ in $\mathcal{S}_{\widehat{\PP}}^{\infty}\times \mathcal{L}_{\widehat{\PP}}^2(\widehat{B})\times \mathcal{L}_{\widehat{\PP}}^2(\widetilde{\mu})$ be the unique solution to the BSDE
\(
Y_t =\xi +\int_t^T f_s(Y_{s-},Z_s,U_s)\, {\rm d}s-\int_t^T Z_s\, {\rm d}\widehat{B}_s-\int_t^T \mskip-10mu \int_E U_s(e)\, \widetilde{\mu}({\rm d}s,{\rm d}e)
\)
under the minimal local martingale measure $\widehat{\PP}$ for the generator 
\begin{equation}
\label{fgenexput}
f_t(y,z,u):=-\frac{|\varphi_t|^2}{2\alpha}+\int_E \frac{\exp (\alpha u(e))-\alpha u(e)-1}{\alpha}\, \zeta (t,e)\, \lambda ({\rm d}e)\,
\end{equation}
which does exist by Theorem~\ref{infinitetheo}. Under $\PP$ the BSDE is of the form
\[
Y_t =\xi +\int_t^T f_s(Y_{s-},Z_s,U_s)-\varphi_sZ_s\, {\rm d}s-\int_t^T Z_s\, {\rm d}B_s-\int_t^T \mskip-10mu \int_E U_s(e)\, \widetilde{\mu}({\rm d}s,{\rm d}e).
\]
To prove optimality by a martingale principle one constructs, cf.\ \cite{Huetal05}, a family of processes $(V^{\theta})_{\theta \in \Theta}$ such that three conditions are satisfied:
{(i)}  $V_t^{\theta}=V_t$ is a fixed  $\mathcal{F}_t$-measurable bounded random variable invariant over $\theta\in \Theta$,
{(ii)} $V_T^{\theta}=-\exp (-\alpha (X_T^{\theta}-\xi ))=-\exp \big(-\alpha \big(x+\int_t^T \theta_s {\rm d}\widehat{B}_s-\xi \big)\big)$, and
{(iii)} $V^{\theta}$ is a supermartingale for all $\theta \in \Theta$ and there exists a $\theta^* \in \Theta$ such that $V_s^{\theta^*}$ ($s\in [t,T]$)
 is a $\PP$-martingale.
Then $\theta^*$ is the optimal strategy and $(V^{\theta^*}_s)_{s\in[t,T]}$ is the value process
of the control problem (\ref{utilitymax1}). Indeed,
 $\mathbb{E}\big(V_T^{\theta}\,\big|\,\mathcal{F}_t\big)\leq V_t^{\theta}=V_t^{\theta^*}=\mathbb{E}\big(V_T^{\theta^*}\,\big|\,\mathcal{F}_t\big)$ for each $\theta \in \Theta$  implies
$v_t(x)=\esssup_{\theta \in \Theta}\mathbb{E}\big(V_T^{\theta}\,\big|\,\mathcal{F}_t\big)=V_t^{\theta^*}$.
An ansatz
 $V^{\theta}= u(X^{\theta}-Y)$ yields
\begin{align*}
V_s^{\theta}&= V_t^{\theta} \exp\Big( \frac{\alpha^2}{2} \int_t^s \Big| \theta_r-Z_r-\frac{\varphi_r}{\alpha}\Big|^2\,{\rm d}r \Big)\, \mathcal{E}(M)_t^s \quad \mbox{for all }s\in [t,T],\quad\text{with} \\
M_t&= -\alpha \int_0^t \theta_r-Z_r\, {\rm d}\widehat{B}_r+\int_0^t \mskip-5mu \int_E \exp (\alpha U_r(e)-1)\, \widetilde{\mu}({\rm d}r,{\rm d}e) \quad \mbox{and} \quad \mathcal{E}(M)_t^s:=\frac{\mathcal{E}(M)_s}{\mathcal{E}(M)_t}.
\end{align*}
Therefore, $V^{\theta}$ is a supermartingale for all $\theta \in \Theta$ and a martingale for $\theta^{*}=Z+\varphi /\alpha$ due to the fact that $\EEE (M)$ is a (local)  martingale of the form
\[
\mathcal{E}(M)_s=\exp \Big( -\frac{\alpha^2}{2}\int_0^s |\theta_u-Z_u-\varphi_u/\alpha |^2\, {\rm d}u \Big)\, \exp \Big(-\alpha \Big(Y_0+\int_0^s\theta_u\, {\rm d}\widehat{B}_u-Y_s\Big)\Big)\,.
\]
Using the boundedness of $Y$, 
one readily obtains by arguments like in \cite{Huetal05,Morlais10} that $\EEE(M)$ is uniformly integrable and hence a martingale  (see\ e.g.\ eqn.~(4.19) in \cite{Becherer06}). This yields

\begin{example}\label{ExplExpU} Let $k=d$ and $\lambda (E)\leq \infty$.
Let $(Y,Z,U)\in \mathcal{S}_{\widehat{\PP}}^{\infty}\times \mathcal{L}_{\widehat{\PP}}^2(\widehat{B})\times \mathcal{L}_{\widehat{\PP}}^2(\widetilde{\mu})$ be the unique solution to the BSDE
$(\xi,f)$ under $\widehat{\PP}$ for generator $f$ from (\ref{fgenexput}).
 Then the strategy  $\theta^*=Z+{\varphi}/{\alpha}$
is optimal for the control  problem (\ref{utilitymax1}) and achieves at any time $t\le T$ the maximal expected exponential utility
$v_t(x)=-\exp (-\alpha (x-Y_t))=V^{\theta^*}_t$.
\end{example}
The exponential utility maximization problem is closely linked to the popular entropic convex risk measure, to which
we will further relate in Example~\ref{counterexplcomparison}.
Moreover, the solution to the utility maximization problem is intimately linked to the indifference valuation (also 
known as reservation price or compensating variation in economics)  for a contingent claim $\xi$ in incomplete markets under exponential 
utility preferences, see \cite{Becherer10uiv}. Indeed, denoting by $Y^{\xi}=Y$ the solution to the JBSDE from
Example~\ref{ExplExpU} for terminal data $\xi$, one can show that $Y^{\xi}-Y^0$ yields the  
 utility indifference valuation process, see \cite{ManiaSchweizer05,Becherer06}.  

\subsubsection{Case with discontinuous risky asset price processes}\label{subsubsec:CaseDiscAssetPrice}
We further illustrate the extend to which results by \cite{Morlais09,Morlais10}, who has pioneered the
stability approach to BSDE with jumps specifically for exponential utility, fit into our framework and demonstrate 
by concrete examples some notable differences in scope in relation to complementary approaches.
To this end, let us consider the same utility problem but now in
a financial market with pure-jump asset price processes, possibly of infinite activity (as e.g.\ in the CGMY model of \cite{CGMY02}), and with constraints on trading strategies. We note that a pure-jump setting 
appears as a natural setup for our JBSDE results, which admit for generators that are (roughly said)  'quadratic' in the $u$-argument but not in $z$-argument, differently from, e.g., \cite{Morlais10,KPZ15,LaevenStadje14,AntonelliMancini16,Yao17}. 

Let $\mu=\mu^L$ be the random measure associated to a pure-jump L\'evy process $L$ with L\'evy measure $\lambda({\rm d}e)$, on $E=\RR^1\setminus\{0\}$. Let $\filt=\filt^L$ be the usual filtration generated by $L$.
 The compensated random measure $\widetilde{\mu}=\widetilde{\mu}^L:=\mu^L-\nu$, with $\nu({\rm d}t,{\rm d}e)=\lambda({\rm d}e){\rm d}t$ of $L$ alone has the weak PRP 
w.r.t.\ the filtration $\mathbb{F}$ (see Example \ref{example_WPRP}.1.). 
 Note that $\mu$ could be of infinite activity, i.e.\ $\lambda(E) \le \infty$, for instance for $L$ being a Gamma process. 
In contrast to the setup of Section \ref{subsubsec:CaseContAssetPrice}, 
we consider now a financial market whose single risky asset prices evolves in a non-continuous fashion, 
being given by a pure-jump process  
\begin{equation*}
	{\rm d}S_t  =S_{t-}\Big(\beta_t{\rm d}t+\int_E\psi_t(e)\widetilde{\mu}({\rm d}t,{\rm d}e\Big)\quad \text{for  $t\in[0,T]$,  with $S_0\in (0,\infty)$,}
\end{equation*}
where $\beta$ is predictable and bounded, and  $\psi>-1$ is $\widetilde{\mathcal{P}}$-measurable, in $L^2(\PP\otimes\lambda\otimes {\rm d}t)\cap L^\infty(\PP\otimes\lambda\otimes {\rm d}t)$ and
satisfies $\int_E\lvert\psi_t(e)\rvert^2\lambda({\rm d}e)\le \text{const.}$ $\PP\otimes {\rm d}t$-a.e.. The set $\Theta$ of admissible trading strategies consists of all $\mathbb{R}$-valued 
predictable $S$-integrable processes $\theta\in L^2(\PP\otimes {\rm d}t)$, such that $\theta_t(\omega)\in C$ for all $(t,\omega)$, for a fixed compact set $C\subset \RR$ of trading constraint
containing $0$.
Interpreting trading strategies $\theta$ as amount of wealth invested into the risky asset yields wealth process $X^{\theta,t,x}$ from initial capital $x$ at time $t$ as
$$
X^{\theta,t,x}_s = X^{\theta,t,x}_t + \int_t^s\theta_u\frac{{\rm d}S_u}{S_{u-}} = x + \int_t^s\theta_u\Big(\beta_u{\rm d}u+\int_E\psi_u(e)\widetilde{\mu}({\rm d}u,{\rm d}e)\Big),\quad s\ge t.
$$
Because of the compactness of $C$ and the fact that $\psi\in L^2(\PP\otimes\lambda\otimes {\rm d}t)\cap L^\infty(\PP\otimes\lambda\otimes {\rm d}t)$, admissible strategies are bounded and for all $\theta\in\Theta$ 
one can verify that
\(
\left\lbrace \exp(-\alpha X^\theta_\tau)|\ \tau\ \text{an $\filt$-stopping time}\right\rbrace
\) {is uniformly integrable};
arguments being like in \cite[Lem.1]{Morlais10}.  Consider the JBSDE 
\begin{equation}\label{eq:JBSDEUMP}
	-{\rm d}Y_t = f(t,U_t){\rm d}t - \int_E U_t(e)\widetilde{\mu}({\rm d}t,{\rm d}e),\quad Y_T=\xi,
\end{equation}
with terminal condition $\xi\in L^\infty(\FF_T)$ and generator $f$ defined pointwise by 
\begin{equation}\label{eq:GenJBSDEUMP}
	f(t,u) := \inf_{\theta\in C}\Big(-\theta\beta_t+\int_Eg_\alpha\big(u(e)-\theta\psi_t(e)\big)\lambda({\rm d}e)\Big),\quad t\in[0,T],
\end{equation}
for the function $g_\alpha:\RR\to\RR$ with $g_\alpha(u) := {({\rm e}^{\alpha u}-\alpha u-1)}/{\alpha}$.
We have the following 
\begin{proposition}
Let $(Y,U)\in \mathcal{S}^{\infty}\times \mathcal{L}^2(\widetilde{\mu})$ be the  unique solution to the JBSDE (\ref{eq:JBSDEUMP}).
Then the strategy  $\theta^*$ such that $\theta^{*}_t$ achieves the infimum in (\ref{eq:GenJBSDEUMP}) for $f(t,U_t)$ is optimal for the control problem (\ref{utilitymax1}) and achieves at any 
  $t\in[0,T]$ the maximal expected exponential utility $v_t(x)=-\exp (-\alpha (x-Y_t))=V^{\theta^*}_t$.
\end{proposition}
\begin{proof}
Using the martingale optimality principle one obtains, like in the cited literature and  analogously to Section \ref{subsubsec:CaseContAssetPrice}, that if $(Y,U)\in\mathcal{S}^\infty\times \LL^2(\widetilde{\mu})$ is a solution to the JBSDE (\ref{eq:JBSDEUMP}) 
then the solution to the utility maximization problem  (\ref{utilitymax1})  is indeed given by $v_t(x) = u(x-Y_t)$ (recall that $u$ denotes the exponential utility function) with the strategy $\theta^*$ where $\theta^{*}_t(\omega)$ achieves the infimum $f(\omega,t,U_t(\omega))$ 
in (\ref{eq:GenJBSDEUMP}) for all $(\omega,t)$ being optimal (it exists by measurable selection \cite{Rockafellar}). 
To complete the derivation of this example, it thus just remains to show that the JBSDE (\ref{eq:JBSDEUMP}) indeed admits a unique solution, with trivial $Z$-component $Z=0$.  
This is shown by applying Theorem~\ref{infinitetheogeneral} and Corollary~\ref{cor:JBSDEZzero} since $\xi\in L^\infty(\mathcal{F}^L_T)$ and the generator $f$ does not have a $z$-argument and is $\filt^L$-predictable in $(t,\omega)$.
It is straightforward, albeit somewhat tedious, to verify that the conditions~1-5 and  $(B_{\gamma^n})$, $n\in \NN$, for Theorem~\ref{infinitetheogeneral} are indeed satisfied for the sequence of $\filt^L$-predictable generators functions 
\begin{equation*}
	f^n(t,u) := \inf_{\theta\in C}\Big(-\theta\beta_t+\int_{A_n}g_\alpha\big(u(e)-\theta\psi_t(e)\big)\lambda({\rm d}e)\Big),
\end{equation*}
where $(A_n)_n$ is a sequence of measurable sets with $A_n\uparrow E$ and $\lambda(A_n)<\infty$ for all $n\in\NN$, typically $A_n = (-\infty,-1/n]\cup[1/n,+\infty).$
Let us refer to \cite[Example~1.32]{KentiaPhD} for details of this verification, but explain here how to proceed further with the proof. 

By the first claim of Theorem~\ref{infinitetheogeneral} (together with Corollary~\ref{cor:JBSDEZzero}) one then gets existence of a solution $(Y,U)\in\mathcal{S}^\infty\times\LL^2(\widetilde{\mu})$ to the JBSDE (\ref{eq:JBSDEUMP}), such that $U\ast\widetilde{\mu}$
is a BMO-martingale. To obtain uniqueness by applying the second claim, we need to check that $f$ satisfies condition $({\rm \bf A}'_{\bm \gamma})$: To this end, we 
define 
\(
	\gamma^{u,u'}_t(e) :=  \sup_{\theta\in C}\gamma^{\theta,u,u'}_t(e)\mathds{1}_{\{u\ge u'\}} +\inf_{\theta\in C}\gamma^{\theta,u,u'}_t(e)\mathds{1}_{\{u< u'\}},
\)
for $\gamma^{\theta,u,u'}_t(e) := \int_0^1g'_\alpha\big(l(u-\theta\psi_t(e))+(1-l)(u'-\theta\psi_t(e))\big){\rm d}l$. Then 
(by Examples~\ref{exgamma} and \ref{exagamma}-2.) 
 we get 
\(
f(t,U_t) - f(t,U'_t) \le \int_E \gamma^{U,U'}_t(e)\big( U_t(e) - U'_t(e)\big)\lambda({\rm d}e)
\)
for all 
$U,U'$ with $\lvert U\rvert_\infty<\infty,\lvert U'\rvert_\infty<\infty$.
Now let $u,u'$ be bounded by $c>0$; Since $g'_\alpha(0)=0$, applying the mean-value theorem to $g'_\alpha$ in the expression of $\gamma^{\theta,u,u'}$ gives 
\(
	\big\lvert \gamma^{\theta,u,u'}_t(e)\big\rvert \le \sup_{|x|\le \tilde{c}}\lvert g''_{\alpha}(x)\rvert\big( \lvert u\rvert + \lvert u'\rvert +\lvert \theta\rvert \lvert \psi_t(e)\rvert \big) 
\)  
{ for all }$\theta\in C$, where $\tilde{c}:=c+\lVert \psi\rVert_\infty \text{diam}(C)$. This implies (for $c=\lvert U\rvert_\infty\vee\lvert U'\rvert_\infty<\infty$)
\begin{equation*}
	\sup_{\theta\in C}\big\lvert \gamma^{\theta,U,U'}_t(e)\big\rvert \le \sup_{|x|\le \tilde{c}}\lvert g''_{\alpha}(x)\rvert\Big( \lvert U_t(e)\rvert + \lvert U'_t(e)\rvert +\text{diam}(C) \lvert\psi_t(e)\rvert\Big).
\end{equation*}
Since $\lvert \inf_\theta\gamma^\theta\rvert \le \sup_\theta\lvert\gamma^\theta\rvert$, $\lvert \sup_\theta\gamma^\theta\rvert \le \sup_\theta\lvert\gamma^\theta\rvert$ and $\psi\ast\widetilde{\mu}$
is a BMO-martingale (as $\psi$ is bounded and $\int_E\lvert\psi_t(e)\rvert^2\lambda({\rm d}e) \le \text{const.},\ \PP\otimes {\rm d}t$-a.e.\ by assumption), then $\gamma^{U,U'}\ast\widetilde{\mu}$ is a BMO-martingale if $U\ast\widetilde{\mu}$ and
$U'\ast\widetilde{\mu}$ are, thanks to $\lvert U\rvert_\infty<\infty$ and $\lvert U'\rvert_\infty<\infty$. Hence $f$ satisfies $({\rm \bf A}'_{\bm \gamma})$. 
\end{proof}

\begin{example}\label{counterexplcomparison} (entropic convex risk measure)
Let us consider the special case  $\beta=\psi\equiv 0$ and $S\equiv 1$, i.e.\ the exponential utility problem {\em without} trading opportunities in a risky asset.
One gets the important and well known example 
of the (dynamic) 
entropic risk measure  $Y_t= (1/\alpha)\log \EE (\exp(\alpha \xi) |\FF_t)$ whose JBSDE description can be identified directly by exponential transformation.
In the setup of the present subsection, this JBSDE is covered by the
application study of \cite{Morlais10} and also by our comparison and wellposedness theorems, without any further conditions on the pure-jump L\'evy process. 
In contrast, let us demonstrate that the same scope is
not already offered by the seminal comparison Thm.2.5 of \cite{Royer06} because her key condition (${\rm \bf A}_{\bm \gamma}$) is not satisfied,  which is also supposed for results in \cite[as Assumpt.6.1 for wellposedness in Thm.6.3(i) and for comparison in Prop.6.4]{KPZ15}  (and is further used for applications in \cite{KPZ16}):
 Indeed for $\alpha:=1$, just consider a
 compound Poisson process $L$ (being of finite activity) with uniformly distributed jump heights, taking $\lambda({\rm d}x):=\mathds{1}_{(0,1]}{\rm d}x$ ($x\in E=\RR\setminus\{0\}$). 
Clearly, the generator $f(u)=\int_E \exp(u(x))-u(x)-1\,\lambda({\rm d}x)=:\int_E g(u(x))\lambda({\rm d}x)$ is not Lipschitz in $u\in L^2(\lambda)$.
With $u^{\pm}(x):=\frac 1 2 (\pm {x}^{-3/2}+nx)\mathds{1}_{(1/n,1]}(x)$ in $L^2(\lambda)$ for $n\in \NN$, 
we get $\int_0^1 ({\rm e}^{u^+}-{\rm e}^{u^-}-u^+ + u^-  ) {\rm d}\lambda \to \infty$ for $n\to \infty $
 while $\int_0^1(u^+-u^-)(x) \cdot(1\wedge |x|){\rm d}x \le \int_0^1 x^{-1/2}{\rm d}x<\infty$ for all $n$,
noting that $ {\rm e}^{u^+}-{\rm e}^{u^-}-u^+ + u^-\ge n x^{-1/2}\mathds{1}_{(1/n,1]}$.
 Thus, there  cannot be  
 constants $c_1\in (-1,0]$, $c_2<\infty$, such that $f(u)-f(v) \le \int_E
  (u-v)(x)\gamma^{u,v}(x)\,\lambda({\rm d}x)$ for all $u,v$, with $c_1(1\wedge|x|)\le \gamma^{u,v}(x) \le c_2(1\wedge|x|)$; 
This shows that condition (${\rm \bf A}_{\bm \gamma}$) in \cite{Royer06}
or Assumption~6.1 in \cite{KPZ15},
 are not satisfied here. Indeed, the  seminal (${\rm \bf A}_{\bm \gamma}$) condition from \cite{Royer06} implies Lipschitz continuity of the generator in $u$.

 Similarly,  a related condition on the jump measure  has been stated  in the assumptions of the main theorem in \cite[Thm.~1, see their inequality (4)]{AntonelliMancini16};  Their article further assumes (like also \cite{Yao17}, for instance)  finite jump activity in that $\lambda(E)<\infty$ holds (in our notation, they write $\eta$ for $\lambda$), what is true in the present example (but is not required for \cite{Royer06}, or our previous sections).
But  in this example there cannot exist constants $c_1\in (-1,0]$, $c_2<\infty$ such that $f(u)-f(u') \le \int_E  (u-u')(x)\gamma^{u,u'}(x)\,\lambda({\rm d}x)$ would hold for all $u,u'\in L^2(\lambda)$, with suitable functions $c_1 \le \gamma^{u,u'}(x) \le c_2 $.
Indeed,  the latter would imply  (using Cauchy-Schwarz inequality) that $f$ is Lipschitz in $u$ on $L^2(\lambda)$, what is not true.
Hence,
the assumptions for Theorem~1 in \cite{AntonelliMancini16} appear not satisfied for the entropic risk example,
and its conditions (with inequality (4) assumed for all $u,u'$)
would
 imply that  the JBSDE generator has to be Lipschitz continuous in $u\in L^2(\lambda$).
Note that the aforementioned problem could be resolved if, e.g., an additional $L^\infty$-bound is available for the $u$-argument, for instance  by an exogenuous a-priori $L^\infty$-estimate on the $U$-component for bounded JBSDE solutions (like from Section~\ref{sec:comp}). We note  that  \cite{AntonelliMancini16} offer results for unbounded
JBSDE solutions (see also \cite{HuLiangTang18} for exponential utility in the continuous case without jumps).

\end{example}

\begin{example}\label{ex:counterexampKPZ}
We continue with the previous entropic risk example, but now take a standard Poisson process instead, i.e.\ $\lambda({\rm d}e) = \updelta_{\{1\}}({\rm d}e)$ as Dirac point measure at the fixed jump height $1$.
Then $L^2(\lambda)$ is isomorphic to $\RR$, and $f(u)=g(u)=\exp(u)-u-1$ for $u\in \RR$.
Obviously $g'$ and $g''$ are of exponential growth (in $u$) and cannot be bounded globally in $u\in \RR$ by an affine function
or by constants; hence Assumption~5.1(iii) for \cite[Thm.5.4]{KPZ15} cannot be satisfied.
Moreover, also Assumption~4.3(iii) for \cite[Thms.4.3, 5.4 and 6.3(ii)]{KPZ15}, noting their Lem.5.4,  appears clearly violated since
 (taking $u'=0$)   there exist no $\psi,c\in \RR$ such that $ |g(u)-\psi u| \le c |u|^2$ for all $u\in \RR$. 
\end{example}

Note that, since the function $(u,\theta)\mapsto f^\theta(\cdot,u) :=-\theta\beta_\cdot+\int_Eg_\alpha\big(u(e)-\theta\psi_\cdot(e)\big)\lambda({\rm d}e)$ is convex, the generator 
constructed as $f=\inf_{\theta\in C}f^\theta(\cdot,u)$  (cf.\ \ref{eq:GenJBSDEUMP}) would be convex in $u$ if the constraint set $C$ were assumed to be convex;  
But for non-convex trading constraints $C$ the generator 
can be non-convex in general. Similar constructions of generators are typical in this application context, 
see e.g.\ \cite[equation~(15)]{LaevenStadje14}.
Some results on JBSDE in the literature use convexity of the generator function but that
 can be restrictive for applications. 
For results in the present paper, convexity is not being assumed. 
Next, we give a  concrete application example where $f$ in (\ref{eq:GenJBSDEUMP}) for the (primal) control problem is indeed non-convex in $u$. The  example  shows, how non-convex constraints can lead to JBSDE generators which are non-convex in $u$. 
%
%
To this end, let us consider simple trading constraints that are non-convex by taking 
$
C:=\{\theta^0,\ldots,\theta^m\}\subset \RR
$ 
as a finite set including the zero $\theta^0:=0$.  Here $f$ of JBSDE (\ref{eq:JBSDEUMP}) becomes
$$
f(t,u)=\inf_{k\in\{0,\ldots,m\}}\Big(-\theta^k\beta_t+\int_Eg_\alpha\big(u(e)-\theta^k\psi_t(e)\big)\lambda({\rm d}e)\Big)\,.
$$
\begin{example} (An application where the generator is not convex and not continuously differentiable) 
Continuing with the above generator $f$, now let us take the  particular case where $\lambda({\rm d}e) = \updelta_{\{1\}}({\rm d}e)$,  i.e.\ $L$ is a standard Poisson process with constant jump height $1$, and let  $\alpha=1$ and $\beta=0$.
Observing that $L^2(\lambda)$ in the case of this simple example is isomorphic to $\RR$, we see that   
\begin{equation}\label{eq:simplegeneratorf}
f(t,u)=\min_{k\in\{0,\ldots,m\}}\left({\rm e}^{(u-\theta^k\psi_t)}-(u-\theta^k\psi_t) -1 \right),
\end{equation}
is clearly  non-convex in $u\in\mathbb{R}$, unless $\psi \equiv 0$ or $C=\{0\}$. Also, we observe that $u\mapsto f(t,u)$ is 
not continuously differentiable in  $u \in \mathbb{R}$ for this application.
But the function $f$ in (\ref{eq:simplegeneratorf}) is still absolutely continuous in $u$
 with density function being strictly greater than $-1$, locally bounded in u from above and locally of linear growth. 
Because this $f$  satisfies condition $({\rm \bf A}_{\rm \bf infi})$, existence and uniqueness for the 
corresponding JBSDE (\ref{eq:JBSDEUMP}) solution  can be  
obtained  by Theorem~\ref{infinitetheo}.

Similarly as before, one can check that in this example the assumptions of \cite[Thm.2.5]{Royer06} and
 \cite[Thm.5.4, Thm.6.3(i)-(ii), Prop.6.4]{KPZ15} for comparison and wellposedness of JBSDE are not satisfied;
The example clearly shows how non-convex constraints can indeed lead to a non-convex $f$, which does not satisfy the conditions  for the JBSDE results of \cite[Thm.A28, Cor.A29, Prop.A30, being used further in the proofs for
Thms.4.3, 4.5, all involving convexity assumption ``(c)'']{LaevenStadje14}.
\end{example}

We proceed next with examples beyond exponential utility, that was the topic in \cite{Morlais09,Morlais10}.

\subsection{Power utility maximization}\label{subsubsec:powut}

Again for the market with stock price dynamics (\ref{eq:continuousStock}), we consider the utility maximization problem
\begin{align} \label{utilitymax2}
v_t(x) =\esssup_{\theta\in \Theta}\, \EE \big(u\big(X_T^{\theta,t,x}\big)\xi\, \big\lvert\,\FF_t\big)=  \frac{1}{\gamma}\esssup_{\theta\in \Theta}\, \EE \big(u\big(X_T^{\theta,t,x} \xi'\big)\, \big\lvert\, \FF_t\big)\,,\quad t\le T,\, x>0,
\end{align}
 for power utility $u(x)=x^{\gamma}/\gamma$ with relative risk aversion $1-\gamma>0$ for $\gamma \in (0,1)$,
with multiplicative liability $\xi$ (alternatively, $\xi':=(\gamma \xi)^{1/\gamma}$ can be interpreted as an unknown future tax rate).  
The wealth process of strategy $\theta$ (denoting fraction of wealth invested) is
 $X^{\theta}_s=X^{\theta,t,x}_s =x+\int_t^s X_u^{\theta}\theta_u\, {\rm d}\widehat{B}_u=x\mathcal{E}(\int \theta {\rm d}\widehat{B})_t^s$ for $s\in[t,T]$, for $\theta \in \Theta$,
with the set $\Theta$ of strategies given by all $\mathbb{R}^d$-valued, predictable, $S$-integrable processes such that $\theta\mal B$ is a ${\rm BMO}(\PP)$-martingale, cf.\ \cite{HeWangYan92}.

\begin{proposition}
\label{propbdblw}
 Let $k=d$.
Assume that there is a solution $(Y,Z,U)\in \set$ to the BSDE $(\xi ,f)$ with $f_t(y,z,u):=(\gamma/(2-2\gamma))\,y\,|\varphi_t+{y}/{z}|^2$ and $\int Z\, {\rm d}B\in {\rm BMO}(\PP )$ and where  $\xi$ is  in $L^{\infty}(\FF_T)$ with $\xi \geq c$ for some $c>0$. Then  $Y\geq c$ holds and $V^{\theta}:=u(X^{\theta})Y$ is a supermartingale for all $\theta$ in $\Theta$ and $V^{\theta^*}$ is a martingale for $\theta^* :=(1-\gamma)^{-1}(\varphi +{Z}/{Y_-})\in \Theta$.
\end{proposition}

\begin{proof}
Clearly, $V^{\theta}$ is adapted.
Kazamaki's criterion  $\mathcal{E}(\int_0^. \gamma \theta_u {\rm d}\widehat{B}_u)$ is an $r$-integrable martingale for some $r>1$. Hence $\sup_{t\leq s\leq T} \mathcal{E}(\int_0^. \gamma \theta_u {\rm d}\widehat{B}_u)_t^s$ is integrable by Doob's inequality.
By
\begin{equation*}
\mathcal{E}(\theta\mal\widehat{B})^{\gamma} = \mathcal{E}(\gamma \theta\mal \widehat{B}) \exp \Big( -\frac{1}{2}\gamma (1-\gamma )\int_0^. |\theta_u|^2\, {\rm d}u\Big)
\leq \mathcal{E}(\gamma \theta\mal\widehat{B})\,,
\end{equation*}
we conclude that $V^{\theta}$ is dominated by $\sup_{t\leq s\leq T}U(X_s^{\theta})|Y|_{\infty}\in L^1(\PP )$.
By It\^o's formula, ${\rm d}V_s^{\theta}$ equals a local martingale plus the finite variation part
\[
u(X_s^{\theta})\left(-f_s(Y_{s-},Z_s,U_s)+\gamma \left(Y_{s-}\Big(\theta_s \varphi_s +\frac{1}{2}(\gamma -1)|\theta_s|^2\Big)+\theta_s Z_s\right)\right){\rm d}s\,.
\]
The latter part is decreasing for all $\theta \in \Theta$ and vanishes at zero for $\theta=\theta^*$. So $V^{\theta}$ 
is a local (super)martingale. Uniform integrability of $V^{\theta}$ yields the (super)martingale property.
By the classical martingale optimality principle of optimal control follows
that $
v_t(x)=\esssup_{\theta\in \Theta} \EE (u(X_T^{\theta}\xi^{1/\gamma} )|\FF_t)$ 
equals $V_t^{\theta^*}=\gamma^{-1}x^{\gamma}Y_t\,,$
and evaluating at $\theta \equiv 0$ yields $\gamma^{-1}x^{\gamma}\EE (\xi |\FF_t)\leq \gamma^{-1}x^{\gamma}Y_t$ and hence $Y\geq c$. Note that $\theta^*$ is in $\Theta$ since $\varphi$ is bounded, $Y$ is bounded away from $0$ and $Z$ is an BMO integrand.
\end{proof}

Let $(Y,Z,U)$ be a solution to the BSDE $(\xi ,f)$ with the above data.
Since a suitable solution theory for quadratic BSDEs with jumps is not available, we transform coordinates by letting
\begin{align}
\widetilde{Y}_t:=Y_t^{\frac{1}{1-\gamma}},\quad \widetilde{Z}_t:=\frac{1}{1-\gamma}Y_{t-}^{\frac{\gamma}{1-\gamma}}Z_t \quad \mbox{and} \quad \widetilde{U}_t:=(Y_{t-}+U_t)^{\frac{1}{1-\gamma}}-Y_{t-}^{\frac{1}{1-\gamma}}\,,  \label{vartrafo}
\end{align}
such that $(\widetilde{Y},\widetilde{Z},\widetilde{U})$ solves the BSDE for data $(\widetilde{\xi},\widetilde{f})$ with $\widetilde{\xi} =\xi^{1/(1-\gamma)}$ and $\widetilde{f}_t(y,z,u)$ given by
\[
\frac{\gamma |\varphi_t|^2}{2(1-\gamma )^2}y+\frac{\gamma}{1-\gamma}\varphi_tz
+\int_E \left(\frac{1}{1-\gamma}\left((u(e)+y)^{1-\gamma}y^{\gamma}-y\right)-u(e)\right)\, \zeta (t,e)\, \lambda ({\rm d}e).
\]
Looking at the proof of Lemma~\ref{basicprops2}, we may assume that $U+Y_-$ coincides pointwise with $Y_-$ or $Y$ so that the above transformation is well-defined due to $Y\geq c$.
In fact, $(\ref{vartrafo})$ gives a bijection between solutions with positive Y-components to the BSDEs $(\xi ,f)$ and $(\widetilde{\xi},\widetilde{f})$  in $\set$.

Next, we show the existence of a JBSDE solution for data $(\xi ,f)$ with $\xi \geq c$ for some $c>0$.
Under the probability measure ${{\rm d}\widetilde{\PP}}:=\mathcal{E}\big(\gamma(1-\gamma)^{-1}\,\varphi\mal B\big)_T \, {{\rm d}\PP}$ the process $\widetilde{B}=B-\int_0^. \gamma(1-\gamma)^{-1} \varphi_t\, {\rm d}t$ is a Brownian motion and the JBSDE
\[
\widetilde{Y}_t =\widetilde{\xi} +\int_t^T \widetilde{f}_s(\widetilde{Y}_{s-},\widetilde{Z}_s,\widetilde{U}_s)\, {\rm d}s -\int_t^T \widetilde{Z}_s\, {\rm d}B_s-\int_t^T \int_E \widetilde{U}_s(e)\, \widetilde{\mu}({\rm d}s,{\rm d}e)
\]
under $\PP$ is of the following form under $\widetilde{\PP}$
\begin{align} \label{BSDEtildeP}
\widetilde{Y}_t =\widetilde{\xi} +\int_t^T \Big( \widetilde{f}_s(\widetilde{Y}_{s-},\widetilde{Z}_s,\widetilde{U}_s)-\frac{\gamma \varphi_s}{1-\gamma}\widetilde{Z}_s\Big) \, {\rm d}s -\int_t^T \widetilde{Z}_s\, {\rm d}\widetilde{B}_s-\int_t^T \int_E \widetilde{U}_s(e)\, \widetilde{\mu}({\rm d}s,{\rm d}e)\,,
\end{align}
noting that $\nu$ is the compensator of $\mu$ under $\PP$ and $\widetilde{\PP}$ as well.
In fact, we have

\begin{lemma} \label{bijection}
Assume $\lambda (E)<\infty$. Then $(\widetilde{Y},\widetilde{Z},\widetilde{U})\in \set$ solves the BSDE $(\widetilde{\xi},\widetilde{f})$ such that $\int \widetilde{Z}\, {\rm d}B$ is in ${\rm BMO}(\PP )$ if and only if $(\widetilde{Y},\widetilde{Z},\widetilde{U})\in \mathcal{S}_{\widetilde{\PP}}^{\infty}\times \mathcal{L}_{\widetilde{\PP}}^2(\widetilde{B})\times \mathcal{L}_{\widetilde{\PP}}^2(\widetilde{\mu})$ solves the BSDE $\big(\widetilde{\xi},\widetilde{f}(y,z,u)-\gamma(1-\gamma)^{-1}\varphi z\big)$ such that $\int \widetilde{Z}\, {\rm d}\widetilde{B}$ is in ${\rm BMO}(\widetilde{\PP} )$.
\end{lemma}
\begin{proof}
Equivalence of $\PP$ and $\widetilde{\PP}$ imply that $\widetilde{Y}\in \mathcal{S}^{\infty}$ if and only if  $\widetilde{Y}\in \mathcal{S}_{\widetilde{\PP}}^{\infty}$. 
Given $\lambda (E)<\infty$, $\widetilde{U}\in \mathcal{L}^2(\widetilde{\mu})$ holds if and only if $\widetilde{U}\in \mathcal{L}_{\widetilde{\PP}}^2(\widetilde{\mu})$ due to the boundedness of $\widetilde{U}$. By \cite[Thm.3.6]{Kazamaki94}, the restriction of the Girsanov transform
\(
\Phi: \mathcal{M}_{\rm c}^{{\rm loc},0}(\mathbb{P})\longrightarrow \mathcal{M}_{\rm c}^{{\rm loc},0}(\widetilde{\PP}),
\)with
\(
M\mapsto M-\langle M,\int_0^. \frac{\gamma \varphi}{1-\gamma}\, {\rm d}B_s\rangle,
\)
onto ${\rm BMO}(\mathbb{P})$ yields a bijection between BMO$(\mathbb{P})$-martingales and BMO$(\widetilde{\PP})$-martingales.
Thus, $\int \widetilde{Z}\, {\rm d}B$ is in ${\rm BMO}(\mathbb{P})$ if and only if $\int \widetilde{Z}\, {\rm d}\widetilde{B}$ is in ${\rm BMO}(\widetilde{\PP} )$ for 
$Z=(1-\gamma )\widetilde{Y}_-^{\gamma}\widetilde{Z}$ since $\Phi \big( \int \widetilde{Z}\, {\rm d}B\big)=\int \widetilde{Z}\, {\rm d}B-\int \gamma(1-\gamma)^{-1}\varphi \widetilde{Z}_s\, {\rm d}s=\int \widetilde{Z}\, {\rm d}\widetilde{B}$. In particular, $\widetilde{Z}\in \mathcal{L}^2(B)$ iff $\widetilde{Z}\in \mathcal{L}_{\widetilde{\PP}}^2(\widetilde{B})$.
\end{proof}

To proceed further, let us note at first that under an equivalent change of measure between $\PP$ and $\widetilde{\PP}$,
 the weak predictable representations property for $(B,\widetilde{\mu})$ under $\PP$ is equivalent to the respective property of 
$(\widetilde{B},\widetilde{\mu})$ under  $\widetilde{\PP}$ for the same filtration, see \cite[Theorem~13.22]{HeWangYan92} and recall Example~\ref{example_WPRP}, Part~2.
According to Corollary~\ref{boundscor}, hence there exists a unique solution $\big(\widetilde{Y},\widetilde{Z},\widetilde{U}\big)\in \mathcal{S}_{\widetilde{\PP}}^{\infty}\times \mathcal{L}_{\widetilde{\PP}}^2(B)\times \mathcal{L}_{\widetilde{\PP}}^2(\widetilde{\mu})$ with positive $Y$-component to the BSDE~(\ref{BSDEtildeP}) with
\[
c^{\frac{1}{1-\gamma}}\exp \Big( -\frac{\gamma |\varphi |_{\infty}^2}{2(1-\gamma )^2} (T-t)\Big)\, \leq\, \widetilde{Y}_t\,\leq\, |\xi |_{\infty}\exp \Big( \frac{\gamma |\varphi |_{\infty}^2}{2(1-\gamma )^2}(T-t)\Big)
\]
such that $\int \widetilde{Z}\, {\rm d}\widetilde{B}$ and $\widetilde{U}*\widetilde{\mu}^{\widetilde{\PP}}$ are BMO($\widetilde{\PP}$)-martingales.
By Lemma~\ref{bijection} and the statement of Proposition~\ref{propbdblw} that every bounded solution to the BSDE $(\xi ,f)$ is bounded from below  away from zero
in $Y\ge c >0$, there is a unique solution $(Y,Z,U)$ in $\set$ with $\int Z\, {\rm d}B\in {\rm BMO}(\PP )$ and it is given by the coordinate transform (\ref{vartrafo}).
We note that $Y$ (resp. $\widetilde Y$) can be interpreted as (dual) opportunity process, see \cite[Sect.4]{Nutz10}. Overall, we obtain the next theorem.
\begin{theorem}
Assume  $\lambda (E)<\infty$ and $d=k$. Let $f_s(y,z,u)= \gamma(2-2\gamma)^{-1}\,y\,\big| \varphi_s+{z}/{y}\big|^2$ and let $\xi \in L^{\infty}(\FF_T)$ with $\xi \geq c$ for some $c> 0$.
Then there exists a unique solution $(Y,Z,U)\in \set$ with $\int Z\, {\rm d}B\in {\rm BMO}(\PP )$ to the BSDE $(\xi ,f)$.
Then the strategy  $\theta_s^{*}=(1-\gamma)^{-1}\Big( \varphi_s+Z_s/Y_{s-} \Big)$ is optimal
for the control problem (\ref{utilitymax2}), achieving 
$v_t(x) =\gamma^{-1}x^{\gamma}Y_t=V^{\theta^*}_t$.
\end{theorem}

\subsection{Valuation by good-deal bounds} \label{subsec:gdprice}
In incomplete  financial markets without arbitrage, there exist infinitely many pricing measures and the bounds
 imposed on valuation solely by the principle of no-arbitrage are typically far too wide
  for applications in practice.
Good-deal bounds
 \cite{CochraneSaaRequejo00} have been introduced in the finance literature to obtain tighter bounds, by ruling out
 not only arbitrage 
but also trading opportunities with overly attractive reward-for-risk ratios, so-called
good deals. See \cite{BechererKentia17a,BechererKentia17b} for extensive references and applications
 under model ambiguity.
In  \cite{CochraneSaaRequejo00,BjoerkSlinko06} good deals have been defined in terms of too favorable instantaneous Sharpe ratios (rate of excess return per unit rate of volatility)  for continuous diffusion processes. 
This has been 
generalized  to a jump-diffusion setup by \cite{BjoerkSlinko06}, who describe good-deal bounds as solutions of nonlinear partial-integro differential equations by using (formal) HJB methods. We  complement their work here 
by a  rigorous, possibly non-Markovian, description by JBSDEs.
See \cite{DelongPelsser15} for a study of a case where the measure $\lambda$ has finite support.

In our setting, the following description of the set $\mathcal{M}^{\rm e}$ of martingale measures is routine.
\begin{proposition}\label{pro:CharactOfMe}
$\mathcal{M}^{\rm e}$ consists of those measures $\QQ\approx \PP$ such that ${{\rm d}\QQ}/{\rm d}\PP=\mathcal{E}\left( \beta\mal B + \gamma\ast\widetilde{\mu}\right)$, where $\gamma>-1$ is a $\widetilde{\mathcal{P}}$-measurable and $\widetilde{\mu}$-integral function, and
$\beta$ is a predictable process with $\int_0^T\abs{\beta_s}^2{\rm d}s<\infty$, satisfying $\beta=-\varphi+\eta$, such that $\eta\in\mathrm{Ker}\ \sigma,\ \PP\otimes {\rm d}t\textrm{-a.e.}$.
\end{proposition}
We will refer to the tuple $(\gamma,\beta)$ for such a density ${\rm d}\QQ/{\rm d}\PP$ as the Girsanov
kernel of $\QQ$ relative to $\PP$. Clearly, our market is incomplete in general as
there exists infinitely many measures in $\mathcal{M}^{\rm e}$ if $\widetilde{\mu}$ is non-trivial or $k<d$.
Bj\"ork and Slinko employed an extended Hansen-Jagannathan  inequality 
\cite[see Sect.2]{BjoerkSlinko06} to bound the
 instantaneous Sharpe ratio by imposing a bound on market prices of risk.
More precisely, Thm.2.3 of \cite{BjoerkSlinko06} showed that  the instantaneous Sharpe ratio ${SR_t}$
in any extension of the market by additional derivative assets (i.e.\ by any local $\QQ$-martingales)  satisfies $\abs{SR_t}\le \Vert(\gamma_t,\beta_t)\Vert_{L^2(\lambda_t)\times\RR^d}$ at any time $t$,
with a (sharp) upper bound in terms of an $L^2$-norm for Girsanov kernels  $(\gamma,\beta)$ of pricing measures in $\mathcal{M}^{\rm e}$, with
$\lambda_t(\omega)({\rm d}e) := \zeta_t(\omega,e)\lambda({\rm d}e)$. As no-good-deal restriction  they therefore impose a bound on the kernels of pricing measures 
\begin{equation}\label{eq:NGDRESTRICTIONWITHBconst}
    \Vert(\gamma_t,\beta_t)\Vert^2_{L^2(\lambda_t)\times\RR^d}= \Vert \gamma_t\Vert^2_{L^2(\lambda_t)} + \abs{\beta_t}^2_{\RR^d} \le K^2,\quad t\le T,
\end{equation}
by some given constant $K>0$.
To complement the analysis of the problem posed by  \cite{BjoerkSlinko06},
we are going to describe the dynamic good deal bounds
rigorously by JBSDEs  in a more general, possibly non-Markovian, setting
with no-good-deal restriction like in (\ref{eq:NGDRESTRICTIONWITHBconst}) but, more generally, we allow 
$K=(K_t)$ to be  a positive predictable bounded process instead of a constant.

To this end, for $K$ as above, let the correspondence (set-valued) process $C$ be given by
\begin{equation}\label{eq:CorrCNGDBoundSR}
 C_t := \left\lbrace(\gamma,\eta)\in L^2(\lambda_t)\times\RR^d\,\Big\lvert\,\ \gamma>-1,\; \eta\in \mathrm{Ker}\, \sigma_t,\;\textrm{and }\;\Vert \gamma\Vert^2_{L^2(\lambda_t)} + \abs{\eta}^2_{\RR^d}+ \abs{\varphi_t}^2_{\RR^d} \le K_t^2\right\rbrace.
\end{equation}
We will write $(\gamma,\eta)\in C$ to denote that $\eta$ is  a  predictable process  and  $\gamma$ is a $\widetilde{\mathcal{P}}$-measurable process   with  $(\gamma_t(\omega),\eta_t(\omega))\in C_t(\omega)$ holding for all $ (t,\omega)\in [0,T]\times\Omega$.
For $(\gamma,\eta)\in C$, we know (cf.\ Example~\ref{exmartingale}.\ref{exmartingale1}) that
$\mathcal{E}\big((-\varphi+\eta)\mal B+\gamma\ast\widetilde{\mu}\big)>0$ is a martingale
that defines a density process of a probability measure $\QQ^{\gamma, \eta}$ which is equivalent to $\PP$. The set of such
probability measures
\begin{equation}\label{eq:QngdDefinition}
              \mathcal{Q}^{\mathrm{ngd}}:=\left\lbrace \QQ^{\gamma,\eta}\,|\,\ (\gamma,\eta)\in C\right\rbrace\subseteq \mathcal{M}^{\rm e}, 
\end{equation}
defines our set of no-good-deal measures. Beyond  boundedness of $\varphi$, assume that $\abs{\varphi_t}_{\RR^d}+\epsilon < K_t$ holds for  for some $\epsilon>0$ for all $\ t\le T$. Then, in particular, the minimal 
martingale measure $\widehat{\PP}=\QQ^{\widehat{\gamma},\widehat{\eta}}$ is in $\mathcal{Q}^{\textrm{ngd}}\neq \emptyset$, with $(\widehat{\gamma},\widehat{\eta})\equiv(0,0)\in C$.
For contingent claims $X\in L^\infty(\PP )$,  the processes
\begin{equation*}
 \pi^u_t(X) := \esssup_{\QQ\in\mathcal{Q}^{\textrm{ngd}}}\EE_{\QQ}(X|\FF_t)\quad\textrm{and}\quad\pi^l_t(X) :=\essinf_{\QQ\in\mathcal{Q}^{\textrm{ngd}}}\EE_{\QQ}(X|\FF_t),\quad t\le T,
\end{equation*}
define the upper and lower good-deal bounds.
Noting $\pi^l_\cdot(X)=-\pi^u(-X)$, we focus on $\pi^u(-X)$. 
One can check that the good-deal bound process satisfies good dynamic properties, e.g.\ time-consistency and recursiveness (cf.\ e.g.\ \cite[Lem.1]{BechererKentia17a}).
By applying the  comparison result of Proposition \ref{comparegeneral}, we are going to obtain $\pi^u(X)$ as the value process $Y$ of a BSDE with terminal condition $X\in L^\infty(\PP )$.
Denoting by $\Pi_t(\cdot)$ and $\Pi^\bot_t(\cdot)$ the orthogonal projections on $\mathrm{Im}\ \sigma_t^{T}$ and $\mathrm{Ker}\ \sigma_t$, we have the following lemma (see \cite[Lemmas~2.14, 2.22]{KentiaPhD} for details).

\begin{lemma}\label{lem:OptProblemExist}
	For $Z\in \mathcal{L}^2(B)$ and $U\in \mathcal{L}^2(\widetilde{\mu})$ there exists $\bar{\eta}=\bar{\eta}(Z,U)$ predictable and $\bar{\gamma}=\bar{\gamma}(Z,U)$ $\widetilde{\mathcal{P}}$-measurable such that
for $\PP\otimes {\rm d}t$-almost all $(\omega,t)\in \Omega\times [0,T]$ holds 
\begin{equation}\label{eq:OptPBSR}
 \bar{\eta}_t\Pi^\bot_t(Z_t)+\int_EU_t(e)\bar{\gamma}_t(e)\zeta_t(e)\lambda({\rm d}e)= \max_{(\gamma,\eta)\in \bar{C}}\Big(\eta_t\Pi^\bot_t(Z_t)+\int_EU_t(e)\gamma_t(e)\zeta_t(e)\lambda({\rm d}e)\Big), 
\end{equation}
where $\bar{C}_t= \left\lbrace(\gamma,\eta)\in L^2(\lambda_t)\times\RR^d\,\Big\lvert\, \gamma\ge-1,\ \eta\in \mathrm{Ker}\ \sigma_t,\ \Vert \gamma\Vert^2_{L^2(\lambda_t)}+\abs{\eta}^2_{\RR^d}\le K^2_t-\abs{ \varphi_t}^2_{\RR^d}\right\rbrace$
is the closure of $C_t$ in $L^2(\lambda_t)\times\RR^d$ for any $t\le T$.
\end{lemma}
To $(\bar{\gamma},\bar{\eta})\in\bar{C}$ of Lemma \ref{lem:OptProblemExist}, we associate a probability measure $\bar{\QQ}\ll\PP$ defined via ${{\rm d}\bar{\QQ}}= \mathcal{E}\left((-\varphi+\bar{\eta})\mal B + \bar{\gamma}\ast\widetilde{\mu}\right)\,{\rm d}\PP$,
which may not be equivalent to $\PP$ as $\bar{\gamma}$ may be $-1$ on a non-negligible set. While $\bar{\QQ}$ might not be in $\mathcal{Q}^{\textrm{ngd}}$
it belongs to the $L^1(\PP)$-closure of $\mathcal{Q}^{\textrm{ngd}}$ in general,
as shown in
\begin{lemma}\label{lem:QbarApprox}
For $Z\in \mathcal{L}^2(B)$ and $U\in \mathcal{L}^2(\widetilde{\mu})$, let $(\bar{\gamma},\bar{\eta})$ be as in Lemma \ref{lem:OptProblemExist}. Define the measures $\bar{\QQ}\ll\PP$ via ${{\rm d}\bar{\QQ}}= \mathcal{E}\big((-\varphi+\bar{\eta})\mal B + \bar{\gamma}\ast\widetilde{\mu}\big) \,{{\rm d}\PP}$
and $\QQ^n := (1/n)\widehat{\PP} + (1-1/n)\bar{\QQ}$ for $n\in\NN$. Then the densities ${\rm d}\QQ^n/{\rm d}\PP$ of the sequence $\left(\QQ^n\right)_{n\in\NN}$ in $ \mathcal{Q}^{\textrm{ngd}}$  converge to the one of $\bar{\QQ}$ in $L^1(\PP)$ for $n\to\infty$. Consequently,
$\pi^u_t(X)\ge \EE_{\bar{\QQ}}(X|\FF_t)$ holds for all $ t\le T$.
\end{lemma}

\begin{proof}
Let $n\in \NN$. Clearly $\QQ^n\approx\PP$. Moreover ${{\rm d}\QQ^n}/{{\rm d}\PP}=Z^n:= (1/n)\widehat{Z} + (1-1/n)\bar{Z}$ with $\widehat{Z}:={{\rm d}\widehat{\QQ}}/{{\rm d}\PP}=\mathcal{E}(-\varphi\mal B)$ and $\bar{Z}:={{\rm d}\bar{\QQ}}/{{\rm d}\PP}$.
It\^o formula then yields $Z^n =\mathcal{E}\big((-\varphi+\eta^n)\mal B+\gamma^n\ast\widetilde{\mu}\big)$ for $\eta^n=\alpha\bar{\eta}$ being predictable and $\gamma^n = \alpha\bar{\gamma}$ is $\widetilde{\mathcal{P}}$-measurable
with $\alpha = {(1-1/n)(\bar{Z}}/{Z^n})\in[0,1)$ thanks to $\widehat{Z}>0$. Therefore $\eta^n\in\mathrm{Ker}\ \sigma$ and $\gamma^n>-1$ due to $\bar{\gamma}\ge-1$. Hence $(\eta^n,\gamma^n)\in C$ and so
$\QQ^n =\QQ^{\gamma^n,\eta^n}$ is in $\mathcal{Q}^{\textrm{ngd}}$. Convergence of $\QQ^n$ to $\bar{\QQ}$ in $L^1(\PP)$ as $n\to\infty$ is straightforward by definition of $\QQ^n$ and this implies $\pi^u_t(X)\ge \EE_{\bar{\QQ}}(X|\FF_t)$ for all $ t\le T$.
\end{proof}
The dynamic good-deal bound $\pi^u(X)$ of $X\in L^\infty(\PP )$ is given by the solution to the JBSDE
\begin{equation}\label{eq:BSDENGD-ExampleSR}
-{\rm d}Y_t =   \Big((-\varphi_t+\bar{\eta}_t)Z_t+\int_EU_t(e)\bar{\gamma}_t(e)\zeta_t(e)\lambda({\rm d}e)\Big){\rm d}t
 -Z_t{\rm d}B_t - \int_E U_t(e)\widetilde{\mu}({\rm d}t,{\rm d}e),\quad t\in[0,T],
\end{equation}
for terminal condition $Y_T=X$, with $\bar{\gamma}=\bar{\gamma}(Z,U)$, $\bar{\eta}=\bar{\eta}(Z,U)$ given by Lemma~\ref{lem:OptProblemExist},
 according to
\begin{theorem}\label{thm:BSDEEssSupHasSol}
 For $X\in L^\infty(\PP )$, the JBSDE 
above
with $(\bar{\gamma},\bar{\eta})$ from (\ref{eq:OptPBSR})  has a unique solution $(Y,Z,U)$ in $\set$. Moreover there exists
  $\bar{\QQ}\ll \PP$
in the $L^1$-closure of $\mathcal{Q}^{\textrm{ngd}}$ (cf.\ Lemma~\ref{lem:QbarApprox}), with density ${{\rm d}\bar{\QQ}}/{{\rm d}\PP}= \mathcal{E}\left((-\varphi+\bar{\eta})\mal B + \bar{\gamma}\ast\widetilde{\mu}\right)$ such that the good-deal bound
satisfies
\begin{equation}\label{eq:ResultNGDBoundfromCompThm}
 \pi^u_t(X) = \esssup_{\QQ\in \mathcal{Q}^{\textrm{ngd}}}\,\EE_{\QQ}(X|\FF_t)=Y_t=\EE_{\bar{\QQ}}(X|\FF_t)\quad \textrm{ for } t\le T.
\end{equation}
\end{theorem}
\begin{proof}
Consider the family of BSDE generator functions defined for $(z,u)\in\RR^d\times L^2(\zeta_\cdot {\rm d}\lambda)$ by $f^{(\gamma,\eta)}(\cdot,z,u) := (-\varphi_\cdot+\eta_\cdot)z+ \int_Eu(e)\gamma_\cdot(e)\zeta_\cdot(e)\lambda({\rm d}e)$ 
and $f^{(\gamma,\eta)}(\cdot,z,u):=0$ elsewhere, 
for $(\gamma,\eta)\in\bar{C}$, where coefficients $\left(\gamma_t(\omega),-\varphi_t(\omega)+\eta_t(\omega)\right)$ of $f^{(\gamma,\eta)}$ are bounded in $L^2(\lambda_t(\omega))\times\RR^d$ by 
$K_f:=\Vert K\Vert_\infty\in (0,\infty)$ for all $(\gamma,\eta)$ and $(t,\omega)$. 
By Lemma \ref{lem:OptProblemExist}, a classical generator function $f$ for the JBSDE (\ref{eq:BSDENGD-ExampleSR}) can be defined such that 
  ($\PP\otimes {\rm d}t$-a.e.)  $f(\cdot,z,u)=\esssup_{(\gamma,\eta)\in \bar{C}}f^{(\gamma,\eta)}(\cdot,z,u)$ for all $(z,u)\in\RR^d\times L^2(\zeta_\cdot {\rm d}\lambda)$ 
and $f$ is (a.e.) Lipschitz continuous in $(z,u)\in\RR^d\times L^2(\lambda_t(\omega)),$ 
with Lipschitz constant $K_f$.  
Indeed, such generator function $f$ can be defined at first (up to a $\PP\otimes {\rm d}t$-nullset) for countably many  $(z,u)$ with $z\in\QQ^d$ and $u\in\{u^n,\, n\in\NN\}$ dense 
subset of $L^2(\lambda)$ and,
noting that $u\,\zeta_t(\omega)^{1/2}$ is in $L^2(\lambda)$ for $u$ in $L^2(\lambda_t(\omega))$,
by  Lipschitz-continuous extension  for all  $(z,u)\in \RR^d\times L^2(\lambda_t(\omega))$. 
By setting $f(t,z,u):=0$  elsewhere (for $u\in L^0(\mathcal{B}(E),\lambda)\setminus L^2(\lambda_t(\omega))$), one 
can define $f$ as Lipschitz continuous even for $(z,u)\in \RR^d\times   L^0(\mathcal{B}(E),\lambda)$.

By classical theory for Lipschitz-JBSDE, equation (\ref{eq:BSDENGD-ExampleSR})
thus has a unique solution $(Y,Z,U)$ in $\mathcal{S}^2\times\mathcal{L}^2(B)\times\mathcal{L}^2(\widetilde{\mu})$ 
which by boundedness of X satisfies $Y\in\mathcal{S}^\infty$ (cf.\ e.g.\ \cite[Prop.3.2-3.3]{Becherer06}). Note that for all $(\gamma,\eta)\in\bar{C}$, clearly $\beta :=-\varphi+\eta$ is bounded and $\int_E\lvert \gamma_t(e)\rvert^2\zeta_t(e)\lambda({\rm d}e)$ is
 bounded uniformly in $t\le T$. Hence by Lemma \ref{lem:RepSolLinBSDEs}, the BSDEs with generators $f^{\gamma,\eta}$ also have unique solutions $(Y^{\gamma,\eta},Z^{\gamma,\eta},U^{\gamma,\eta})\in\set$, which satisfy
$Y^{\gamma,\eta}_t=\EE_{\QQ^{\gamma,\eta}}(X|\FF_t)$, $ \QQ^{\gamma,\eta}\textrm{-a.s.,}\ t\le T$. Since $f=f^{\bar{\gamma},\bar{\eta}}$, we also have $Y_t = \EE_{\bar{\QQ}}(X|\FF_t)$, $\bar{\QQ}\textrm{-a.s.}$.
By Lemma \ref{lem:QbarApprox} holds $\pi^u_t(X)\ge \EE_{\bar{\QQ}}(X|\FF_t)$, $ \bar{\QQ}\textrm{-a.s.}$, for all $ t\le T$. To complete the proof, we  show that
$\pi^u_t(X)\le Y_t$. For all $(\gamma,\eta)\in C$ (defining $\QQ^{\gamma,\eta}\in\mathcal{Q}^{\textrm{ngd}}$) we have
that $
 f_t(Z_t,U_t)=f^{\bar{\gamma},\bar{\eta}}_t(Z_t,U_t) $ dominates $f^{\gamma,\eta}_t(Z_t,U_t)$  for a.e.\ $t\le T$.
Noting that $f^{\gamma,\eta}$ are Lipschitz in $(z,u)$ 
with (uniform) Lipschitz constant $K_f$ 
and
\begin{equation}\label{eq:IneqonGenwithGammadiffFromRoyer}
 f^{\gamma,\eta}_t(Z^{\gamma,\eta}_t,U_t) -f^{\gamma,\eta}_t(Z^{\gamma,\eta}_t,U^{\gamma,\eta}_t) =\int_E\gamma_t(e)(U_t(e)-U^{\gamma,\eta}_t(e))\zeta_t(e)\lambda({\rm d}e),\quad t\le T\,,                             
\end{equation}
with $\mathcal{E}\left((-\varphi+\eta)\mal B+\gamma\ast\widetilde{\mu}\right)$ being a martingale (see Example \ref{exmartingale}), one can apply comparison as in
Proposition \ref{comparegeneral} to get $Y_t\ge Y^{\gamma,\eta}_t$, $\PP\textrm{-a.s.},\ t\le T,\ (\gamma,\eta)\in C$.
Hence $Y_t\ge \esssup_{(\gamma,\eta)}Y^{\gamma,\eta}_t=\pi^u_t(X)$, $t\le T$, for $(\gamma,-\varphi+\eta)$ ranging over all Girsanov kernels of measures $\QQ\in\mathcal{Q}^{\textrm{ngd}}$.
\end{proof}

\small


\end{document}